\documentclass[12pt,reqno]{amsart}

\usepackage{amsmath,amsfonts,amsthm,amssymb,amscd, dsfont,tipa,mathtools,listings, rotating, multirow, bbm}
\usepackage{hyperref}
\usepackage{color}
\usepackage{comment}
\usepackage{seqsplit}
\usepackage[shortlabels]{enumitem}
\usepackage{soul}
\usepackage{mathrsfs}
\usepackage[title]{appendix}
\usepackage{constants}
\newconstantfamily{abcon}{symbol=c}
\usepackage{graphicx}%
\usepackage{xcolor}%
\usepackage{textcomp}%
\usepackage{manyfoot}%
\usepackage{booktabs}%
\usepackage{algorithm}%
\usepackage{algorithmicx}%
\usepackage{algpseudocode}%
\usepackage{listings}%

\newcommand{\bburl}[1]{\textcolor{blue}{\url{#1}}}

\newcommand\Item[1][]{%
	\ifx\relax#1\relax  \item \else \item[#1] \fi\abovedisplayskip=0pt\abovedisplayshortskip=0pt~\vspace*{-\baselineskip}}
\renewcommand{\tilde}{\widetilde}

\newcommand\be{\begin{equation}}
    \newcommand\ee{\end{equation}}
\newcommand\bi{\begin{itemize}}
    \newcommand\ei{\end{itemize}}
\newcommand\ben{\begin{enumerate}}
    \newcommand\een{\end{enumerate}}

\newtheorem{thm}{Theorem}[section]

\newtheorem{lem}[thm]{Lemma}
\newtheorem{prop}[thm]{Proposition}

\newtheorem{defi}[thm]{Definition}

\theoremstyle{remark}
\newtheorem*{rek}{Remark}

\newcommand{\re}{\mathbb{R}}

\renewcommand{\tilde}{\widetilde}

\DeclareMathOperator{\GL}{GL}

\DeclareMathOperator{\Real}{Re}

\renewcommand{\tilde}{\widetilde}
\renewcommand{\bar}{\overline}
\renewcommand{\epsilon}{\varepsilon}

\renewcommand{\Re}{\mathrm{Re}}

\newcommand{\N}{\mathrm{N}}

\newcommand{\A}{\mathbb{A}}

\newcommand{\CC}{\mathbb{C}}
\newcommand{\R}{\mathbb{R}}
\newcommand{\Q}{\mathbb{Q}}
\newcommand{\ka}{\mathfrak{a}}
\newcommand{\kp}{\mathfrak{p}}
\newcommand{\kq}{\mathfrak{q}}
\newcommand{\kn}{\mathfrak{n}}

\newcommand{\cO}{\mathcal{O}}

\makeatletter
\let\@wraptoccontribs\wraptoccontribs
\makeatother

\numberwithin{equation}{section}

\usepackage{lineno}
\usepackage{lipsum}

\begin{document}

	\title[Metric approach to zero-free regions]{A Metric Approach to Zero-Free Regions for $L$-Functions}

  \author{Nawapan Wattanawanichkul}
	
	\address[Nawapan Wattanawanichkul]{1409 West Green Street, University of Illinois Urbana-Champaign, Urbana, IL 61801, USA}
	\email{nawapan2@illinois.edu}

	\begin{abstract}
    For integers $m,m' \ge 1$, let $\pi$ and $\pi'$ be cuspidal automorphic representations of $\GL(m)$ and $\GL(m')$, respectively. We present a new proof of zero-free regions for $L(s, \pi)$ and for $L(s, \pi \times \pi')$  under the assumption that $\pi, \pi'$ or $L(s,\pi \times \pi')$ is self-dual. Our approach builds on ideas of ``pretentious'' multiplicative functions due to Granville and Soundararajan (as presented by Koukoulopoulos)  and the notion of a positive semi-definite family of automorphic $L$-functions due to Lichtman and Pascadi.
    \end{abstract}
        
    
    \maketitle
    
	\section{Introduction}
	
	\subsection{Notation and motivation} In 1896, Hadamard and de la Vall\'ee Poussin, as part of their independent proofs of the prime number theorem, established that the Riemann zeta function $\zeta(s)$ is nonzero on $\Real(s)=1$.  In 1899, de la Vallée Poussin further showed that 
	$\zeta(s)$ remains nonzero in a narrow region to the left of $\Real(s)=1$, where the width at height 
	$t$ is proportional to $(\log t)^{-1}$. Such a region is referred to as a \textit{classical zero-free region} of $\zeta(s)$. 
	
	To explain their argument, we let $\sigma, t \in \mathbb{R}$ and  $s= \sigma+it$. For $\sigma >1$, we know that
\begin{equation}\label{eq:euler}
		{\zeta(\sigma +it)}^{-1} = \prod_{p}\left(1- {p^{-(\sigma+it)}}\right).
	\end{equation}
	For $t \neq 0$, suppose that $\zeta(1+it) = 0$. By analyticity of $\zeta(s)$, we must have that for some $a\in \mathbb{C}$, $\zeta(\sigma+it) \sim a(\sigma-1)$ as $\sigma \to 1^+$.
	The only way that the product in \eqref{eq:euler} can tend to infinity as $\sigma \to 1^{+}$ is when $p^{-it}$ frequently points towards $-1$.  Consequently, $p^{-2it}$ must often point towards $1$. Thus, via \eqref{eq:euler}, $\zeta(s)$ should have a pole at $1+2it$, which is impossible as $\zeta(s)$ has only one simple pole at $s=1$ but $t \neq 0$. 
	
	Motivated by the theory of ``pretentious'' multiplicative functions developed by Granville and Soundararajan (see \cite{GS}),  Koukoulopoulos in \cite[Theorem 8.3]{D} formalized the idea of Hadamard and de la Vall\'ee Poussin through the following family of metrics. For any positive integer $n$, let $f(n)$ and $g(n)$ be multiplicative functions taking values in the unit circle. For any $\sigma > 1$, he defined 
	\begin{equation}\label{def:distance}
		{\mathbb{D}}_{\sigma}(f,g)^2 := \frac{1}{2} \sum_{p}\sum_{k=1}^{\infty} \frac{|f(p^k)-g(p^k)|^2\log p }{p^{k\sigma}},  
	\end{equation}
	where $\sigma$ serves as a parameter to be optimized later.  
	
	Let $\mu(n)$ be the M\"obius function. Koukoulopoulos considered the triangle inequality
    \[
    {\mathbb{D}}_{\sigma}(n^{-i\gamma},n^{i\gamma}) \le {\mathbb{D}}_{\sigma}(n^{-i\gamma},\mu(n)) +{\mathbb{D}}_{\sigma}(\mu(n),n^{i\gamma}) = 2{\mathbb{D}}_{\sigma}(\mu(n),n^{i\gamma}),
    \]
    which implies, after squaring both sides, that
	\begin{equation}
    \label{eq:triangle-intro}
        {\mathbb{D}}_{\sigma}(n^{-i\gamma},n^{i\gamma})^2\le 4{\mathbb{D}}_{\sigma}(\mu(n),n^{i\gamma})^2.
	\end{equation}
	The above inequality  rigorously demonstrates that if $p^{i\gamma}$ often points towards $-1$, then $p^{2i\gamma}$  frequently points towards $1$. Each distance in \eqref{eq:triangle-intro} can be expressed as a linear combination of the logarithmic derivatives of $L$-functions, up to an additive constant  $O(1)$. Specifically, \eqref{eq:triangle-intro} can be rewritten as
	\begin{equation*}
	    0 \le -3\frac{\zeta'}{\zeta}(\sigma) -4\Real\Big(\frac{\zeta'}{\zeta}(\sigma+i\gamma)\Big)-\Real\Big(\frac{\zeta'}{\zeta}(\sigma+2i\gamma)\Big)+O(1),
	\end{equation*}
	which recasts the \textit{3-4-1} argument by Mertens. 
	
    Let $F$ be a number field with adele ring $\mathbb{A}_F$,
    $\cO_F$ the ring of integers of $F$, and $\N=\N_{F/\Q}$ the norm defined on nonzero ideals $\ka$ of $\cO_F$ by $\N\ka=|\cO_F/\ka|$.
    Let $\mathfrak{F}_{m}$ denote the set of cuspidal automorphic representations $\pi$ of $\GL_{m}(\mathbb{A}_F)$ whose central character $\omega_{\pi}$ is unitary.   
    For convenience, we also define 
    \begin{equation}\label{def:f*}
        \mathfrak{F}_{m}^* := \{\pi \in \mathfrak{F}_m: \omega_{\pi} \text{ is trivial on the diagonally embedded positive reals.}\}.
    \end{equation} 
    In this paper, we restrict ourselves to the case where $\pi \in \mathfrak{F}_{m}^*$ and $\pi' \in \mathfrak{F}_{m'}^*$ and adapt the family of metrics in \eqref{def:distance} to establish zero-free regions for $L(s,\pi)$ and $L(s,\pi\times\pi')$.

\subsection{Statements of results}
Recall the notation in the introduction and refer to Section~\ref{sec:L-functions} for more discussion of $L$-functions.  
Let $\pi \in \mathfrak{F}_m^*$ and $\pi' \in \mathfrak{F}_{m'}^*$, and let $\mathfrak{C}(\pi)\geq 3$ denote the analytic conductor of $\pi$. We now present the main theorems of this paper. 
	
	\begin{thm}
		\label{thm1:ZFR-cuspforms}
		Let $m \ge 1$ be an integer and $\pi \in \mathfrak{F}^*_m$. Let $\sigma,t\in\mathbb{R}$ and $s= \sigma+it$. Then $L(s, \pi)$ has no zero in the region
		\begin{equation*}
			\begin{aligned}
				\sigma 
				\ge 1- \frac{1}{33(2m+3)\log(\mathfrak{C}(\pi) (3+|t|)^{m[F:\Q]})},
			\end{aligned}
		\end{equation*}
		except possibly for one simple real zero, in which case $\pi$ is self-dual. Moreover, when $\pi$ is self-dual, the exponent $m[F:\Q]$ of $3+|t|$ can be reduced to $[F:\Q]$.
	\end{thm}
	\begin{thm}
		\label{thm1:ZFR-rankin-selberg}
		Let $m, m' \ge 1$ be integers, and  let $\pi \in \mathfrak{F}^*_m$ and $\pi' \in \mathfrak{F}^*_{m'}$.
		Let $\sigma,t\in\mathbb{R}$ and $s= \sigma+it$. If  $\pi\neq\tilde{\pi}$, $\pi'=\tilde{\pi}'$, and
		\begin{equation*}
			\begin{aligned}
				\sigma 
				\ge 1- \frac{1}{28(m+m')\log(\mathfrak{C}(\pi) \mathfrak{C}(\pi')(3+|t|)^{m[F:\Q]})},
			\end{aligned}
		\end{equation*}
		then $L(\sigma+it,\pi \times \pi')\neq 0$. 
	\end{thm}

	\begin{thm}\label{thm2:ZFR-rankin-selberg} 
	Let $m, m' \ge 1$ be integers, and  let $\pi \in \mathfrak{F}^*_m$ and $\pi' \in \mathfrak{F}^*_{m'}$.
		Let $\sigma,t\in\mathbb{R}$ and $s= \sigma+it$. 
		If $L(s, \pi \times \pi') = L(s, \tilde{\pi} \times \tilde{\pi}')$, then
    $L(s, \pi \times \pi')$ has at most one zero (necessarily real and simple) in the region
		\[
		\sigma 
		\ge 1- \frac{1}{66(m+m')\log(\mathfrak{C}(\pi) \mathfrak{C}(\pi')(3+|t|)^{m[F:\Q]})}.
		\]
	\end{thm}

    \subsection{Novelty of the present work}

    We now explain our motivation for adapting the metric in \eqref{def:distance} and highlight the novelty of our work. As in the case of $\zeta(s)$, to prove a zero-free region for $L(s, \pi)$, we assume it has a zero at $s = 1 + it$ with $t \neq 0$. Let $\lambda_{\pi}(\mathfrak{n})$ be the Dirichlet coefficient of $L(s, \pi)$. For a prime ideal $\kp$, we expect $\N\kp^{-it}$ to often point toward $-\lambda_{\pi}(\kp)$, suggesting the study of \begin{equation} \label{eq:lambda} \sum_{\mathfrak{p}} \sum_{k=1}^{\infty} \frac{\lambda_{\pi}(\kp^{k}) \log \N\kp}{\N\kp^{ks}}, \end{equation} for $\Re(s) > 1$, which arises naturally from \eqref{def:distance}.

    Assuming progress towards the generalized Ramanujan conjecture (GRC), the sum in \eqref{eq:lambda} can be approximated by $-(L'/L)(s, \pi) + O(1)$, where the implied constant depends on $m$. A similar approximation holds for $L(s, \pi \times \pi')$, but now requires progress towards GRC that is not yet proved in full generality to handle prime powers. To avoid this, we use the observation that the Dirichlet coefficient of $-(L'/L)(s, \pi)$ at a prime ideal $\mathfrak{p}$, denoted $a_{\pi}(\mathfrak{p})$, equals $\lambda_{\pi}(\mathfrak{p})$. Thus, instead of \eqref{eq:lambda}, it is more natural to study
\begin{equation} \label{eq:a-pi}
    \sum_{\mathfrak{p}} \sum_{k=1}^{\infty} \frac{a_{\pi}(\mathfrak{p}^{k}) \log \mathrm{N}\mathfrak{p}}{\mathrm{N}\mathfrak{p}^{ks}},
\end{equation}
where $ \Re(s) > 1 $, which leads naturally to the modified metric in \eqref{eq:gen-distance}. This shift in perspective not only circumvents the need for progress towards GRC, but also underpins our new approach to proving zero-free regions for both standard and Rankin--Selberg $L$-functions. In addition, our approach improves upon earlier results in several aspects:
    \begin{enumerate}
        \item Brumley~\cite[Theorem A.1]{Humphries-Brumley} was the first to establish a classical zero-free region for $L(s, \pi \times \pi')$ under the assumption that at least one of $\pi$ or $\pi'$ is self-dual. However, his result does not include an explicit constant and exhibits a weaker dependence on the degrees of $\pi$ and $\pi'$ compared to Theorems~\ref{thm1:ZFR-rankin-selberg} and \ref{thm2:ZFR-rankin-selberg}. Goldfeld and Li \cite[Theorem 1.2]{GL}, in contrast, established a zero-free region for $L(s, \pi \times \tilde{\pi})$ of the form  
        \[
        \sigma > 1 - \frac{c}{(\log(|t| + 2))^5},
        \]
        which is narrower than the region in Theorem~\ref{thm2:ZFR-rankin-selberg} and is proven only in the $t$-aspect. 
        Extending upon their work, Humphries and Brumley~\cite[Theorem 1.9]{Humphries-Brumley} employed sieve methods to widen the Goldfeld--Li region to a zero-free region of the form 
        \[
        \sigma > 1 - \frac{c}{\log(|t| + 2)}.
        \]
        More recently, Humphries and Thorner~\cite[Theorem 2.1]{Humphries-Thorner} obtained an unconditional refinement of this region with improved uniformity in $\pi$. Along a parallel axiomatic line of inquiry, Leung \cite[Theorem 2.1]{Leung} established a classical zero-free region for a class of $L$-functions whose Dirichlet logarithms have non-negative coefficients. While this framework encompasses self-dual convolutions such as $L(s,\pi\times\tilde{\pi})$, it does not apply, for example, to $L(s,\pi\times\pi')$ considered in Theorem~\ref{thm1:ZFR-rankin-selberg}.
        \item 
        Theorem~\ref{thm2:ZFR-rankin-selberg} is new, except in the cases where (i) both $\pi$ and $\pi'$  are self-dual or (ii) $\pi' = \tilde{\pi}$, which have been previously established—for instance, by Brumley \cite[Theorem A.1]{Humphries-Brumley} and Moreno \cite[Theorem 3.3]{moreno} in the former case, and by Humphries--Thorner \cite[Theorem 2.1]{Humphries-Thorner} in the latter. 
        Furthermore, taking  $\pi'=\mathbbm{1}$,
        Theorems~\ref{thm1:ZFR-rankin-selberg} and \ref{thm2:ZFR-rankin-selberg} together imply Theorem~\ref{thm1:ZFR-cuspforms}.

        \item
            The uniformity in $m$ and $m'$ in Theorems~\ref{thm1:ZFR-cuspforms}--\ref{thm2:ZFR-rankin-selberg} improves upon or matches all prior results. 
           Degree dependence is crucial in certain problems in analytic number theory. For instance, in \cite{Sato-Tate-2}, an effective version of the Sato--Tate conjecture for holomorphic cuspidal newforms relies on uniform prime number theorems for the $L$-functions of their symmetric $m$-th powers, for all $m$ up to some integer $M$. The quality of the error term depends on the choice of $M$, which in turn relies on the degree dependence of zero-free regions for these $L$-functions.

        \item An alternative proof of Theorem~\ref{thm1:ZFR-cuspforms} in \cite[Theorem 5.10]{IK} requires a non-negative product of $L$-functions where the multiplicity of a zero $\rho$ in a certain region exceeds the number of poles at $s = 1$. Such constructions could be difficult without the notion of isobaric sums, which inherently produces a non-negative linear combination of logarithmic derivatives of $L$-functions with a square number of terms.
        Our method, in contrast, yields a different range of non-negative linear combinations. 
        For example, it recovers Mertens' \textit{3-4-1} argument, which clearly has $8$ terms. Our approach, which extends the method in \cite[Theorem~8.3]{D}, thus brings the standard and Rankin--Selberg $L$-functions settings in line with the approach of Hadamard and de la Vall\'ee Poussin.

         \item  The proof in \cite[Theorem 5.10]{IK} requires a minor adjustment due to an oversight in \cite[ Lemma 5.9]{IK}. Specifically, their argument accounts only for the contribution from the poles of the non-negative product of $L$-function under consideration at $s =1$ but overlooks potential contributions from poles at $s= 1+i\gamma$ for some $\gamma \neq 0$, which can be significant when $\gamma$ is small. We address this issue and provide a more complete analysis in Section~\ref{subsubsec:Thm1-case2}.  
    \end{enumerate}

\subsection{Outline of the metric proof architecture}\label{subsec:outline-proof}
Our strategy for establishing a zero-free region for the Rankin--Selberg $L$-function $L(s, \pi \times \pi')$ is based on a metric framework that systematically replaces the non-negative trigonometric functions utilized in classical proofs.

\begin{enumerate}
    \item First, we introduce our family of distances in \eqref{eq:gen-distance}, which is a modified version of Koukoulopoulos's family of ``pretentious'' distances in \eqref{def:distance}. In this variant, the usual Dirichlet coefficients of $L$-functions are replaced by those coming from the logarithmic derivative of the $L$-function. This reformulation helps eliminate any dependence on unproven progress toward GRC.
    
    \item Second, we invoke the triangle inequality for a suitable choice of elements within our family. From an inequality of the form $d_1 \le d_2 + d_3$, we obtain the equivalent expression $0 \le d_2 + d_3 - d_1$, which yields a non-negative linear combination of logarithmic derivatives of $L$-functions. This construction serves as a direct analogue of the classical $3-4-1$ trigonometric identity. When the level of $\pi$ or $\pi'$ exceeds $1$, i.e., $\N\mathfrak{q}_{\pi} > 1$ or $\N\mathfrak{q}_{\pi'} >1$, the ramified primes contribute additional terms to this linear combination, but they can be controlled using Lemma~\ref{lem:general-error}.
    
    \item Let $\rho = \beta + i\gamma$ be a nontrivial zero of $L(s, \pi \times \pi')$. By applying Lemma~\ref{lem:stirling} to our previously obtained linear combination and isolating the contribution of $\beta$, we arrive at an inequality of the form
    \begin{equation}\label{eq:desired form-bino}
    0 \le \frac{c_P}{\sigma - 1} - \frac{c_Z}{\sigma - \beta} + E(\pi, \pi', \gamma),
    \end{equation}
    for some constants $c_P, c_Z \ge 0$ and an error term $E(\pi, \pi', \gamma)$ that depends on $\pi, \pi'$, and $\gamma$. A classical zero-free region then follows, provided that our choice of elements in the triangle inequality ensures that 
    \begin{equation}\label{eq:c_P-c_Z-condition}
    0 \le c_P < c_Z.
    \end{equation} 
    This final step follows standard techniques, paralleling the treatments in \cite{D, IK}.
\end{enumerate}
As an alternative approach to managing the ramified prime contributions, in Section~\ref{sec:proof-level>1} we utilize the framework of positive semi-definite families of Rankin--Selberg $L$-functions introduced by Lichtman and Pascadi \cite{LP}. This machinery yields the same non-negative linear combination of logarithmic derivatives of $L$-functions while entirely eliminating the appearance of ramified prime terms. Consequently, this alternative path provides the clean, classical zero-free regions presented in Theorems~\ref{thm1:ZFR-cuspforms}--\ref{thm2:ZFR-rankin-selberg}.

\subsection{Limitations of the metric approach}\label{subsec:limitations-intro}

For a general Rankin--Selberg $L$-function $L(s,\pi\times\pi')$ without self-duality assumptions, establishing a classical zero-free region unconditionally remains a formidable open problem. Suppose that $\pi$, $\pi'$, and $L(s,\pi\times\pi')$ are all non-self-dual. If we implement our metric framework as described in Section~\ref{subsec:outline-proof},  we encounter a structural barrier showing that a zero-free region in this asymmetric setting cannot be obtained, regardless of the choice of triangle inequality. 
In fact, every linear combination arising from an arbitrary triangle inequality can be shown to satisfy $c_P \ge c_Z$. It follows that the criterion in \eqref{eq:c_P-c_Z-condition} is unattainable, and therefore this framework is incapable of producing a zero-free region.
See Section~\ref{sec:limitation-proof} for more details.

Nevertheless, the limitations of the metric framework in this asymmetric setting suggests a natural direction for future inquiry. Specifically, one might exploit the framework of Lichtman and Pascadi \cite{LP}, as presented in Section~\ref{sec:proof-level>1}, which generates a broader family of non-negative linear combinations of logarithmic derivatives than our metric approach. Although it is far from immediate, the answer to whether their framework can yield a classical zero-free region in the absence of self-duality is ultimately negative. We will address this question in full detail in a forthcoming work.

Alternatively, several recent works have established non-classical zero-free regions for general Rankin--Selberg $L$-functions within modern frameworks. For instance, Harcos and Thorner \cite{Harcos-Thorner}, inspired by Siegel's ineffective lower bounds for Dirichlet $L$-functions, obtained a new type of zero-free region for all $\mathrm{GL}(1)$-twists of $\mathrm{GL}(n) \times \mathrm{GL}(n')$ convolutions. More recently, they refined this framework in~\cite{Tatuzawa} to establish a full Tatuzawa-type analogue for these twisted Rankin--Selberg $L$-functions, thereby rendering the underlying constants effectively computable up to a single possible exceptional twist.

	
	
	\subsection*{Acknowledgements} 
    The author thanks Jesse Thorner for suggesting this project and providing valuable guidance and support throughout its development. The author also thanks Dimitris Koukoulopoulos and Alexandru Pascadi for helpful comments on the manuscript. The author is also grateful to the anonymous referee for their insightful feedback and valuable suggestions, which have significantly enhanced both the clarity and quality of the manuscript. All numerical computations were carried out using Mathematica 14.

    \subsection*{Declarations}

    The author has no competing interest and did not receive any specific funding for this work.

    \subsection*{Outline of the paper}
	In Section~\ref{sec:L-functions}, we review the fundamental properties of both standard and Rankin--Selberg $L$-functions. In Section~\ref{sec:gen-strategy}, we present preliminary computations and establish useful bounds that will be used throughout the paper. In Section~\ref{sec:proof-level1}, we prove Theorems~\ref{thm1:ZFR-cuspforms}--\ref{thm2:ZFR-rankin-selberg} under the assumption that $\N\kq_{\pi} = \N\kq_{\pi'}=1$. In Section~\ref{sec:proof-level>1}, we extend the proofs of Theorems~\ref{thm1:ZFR-cuspforms}--\ref{thm2:ZFR-rankin-selberg} in Section~\ref{sec:proof-level1} to the case where $\N\kq_{\pi}>1$ or $\N\kq_{\pi'}>1$.

	
	
	\section{Properties of \texorpdfstring{$L$}{L}-functions}
	\label{sec:L-functions}





    Let $F$ be a number field with adele ring $\A_F$. Let $D_F$ be the absolute discriminant of $F$, $\cO_F$ the ring of integers of $F$, and $\N=\N_{F/\Q}$ the norm defined on nonzero ideals $\ka$ of $\cO_F$ by $\N\ka=|\cO_F/\ka|$. For a place $v$ of $F$, let $v\mid\infty$ (resp. $v\nmid\infty$) denote that $v$ is archimedean (resp. non-archimedean), and let $F_v$ be the corresponding completion of $F$. Each $v\nmid\infty$ corresponds with a prime ideal $\kp$ of $\cO_F$. The properties of $L$-functions given here rely on \cite{GJ,JPS,MW}.  In our use of $f=O(g)$ or $f\ll g$, the implied constant is always absolute and effectively computable, unless specified otherwise.
    

    Let $\pi\in\mathfrak{F}_{m}$.  Recall $\mathfrak{F}_{m}^*$ in \eqref{def:f*}.  There exists a unique pair $(\pi^*,t_{\pi})\in\mathfrak{F}_{m}^*\times\mathbb{R}$ such that $\pi=\pi^*\otimes|\cdot|^{it_{\pi}}$, in which case $L(s,\pi)=L(s+it_{\pi},\pi^*)$.  Therefore, in our discussion of standard $L$-functions and Rankin--Selberg $L$-functions, it suffices to restrict to $\mathfrak{F}_{m}^*$.

    \subsection{Standard \texorpdfstring{$L$}{L}-functions}
    \label{subsec: standardL}

If $\pi\in\mathfrak{F}_m^*$, then for each place $v$, there exists an irreducible admissible representation $\pi_v$ of $\GL_m(F_v)$, with $\pi_v$ ramified for at most finitely many $v$, such that $\pi$ is a restricted tensor product $\otimes_v \pi_v$. When $v\nmid\infty$ and $\kp$ corresponds with $v$, we write $\pi_v$ and $\pi_{\kp}$ interchangeably. For each prime ideal $\kp$, there exist $m$ Satake parameters $(\alpha_{j,\pi}(\kp))_{j=1}^{m}$ such that the local $L$-function $L(s,\pi_{\kp})$ is given by
\[
L(s, \pi_{\kp}) = \prod_{j=1}^m \frac{1}{1-\alpha_{j,\pi}(\kp)\N \kp^{-s}} = \sum_{j=0}^{\infty} \frac{\lambda_{\pi}(\kp^j)}{\N\kp^{js}}.
\]
The standard $L$-function $L(s,\pi)$ attached to $\pi$ is 
\[
L(s,\pi)=\prod_{\kp} L(s, \pi_{\kp}) 
= \sum_{\mathfrak{n}}\frac{\lambda_{\pi}(\mathfrak{n})}{\N\mathfrak{n}^{s}},
\]
which converges absolutely for $\Re(s) > 1$.
If $v\mid \infty$, there exist  $m$ Langlands parameters at $v$, namely, $\mu_{1,\pi}(v), \mu_{2,\pi}(v), \ldots, \mu_{m,\pi}(v) \in \mathbb{C}$, from which we define
\[
L_{\infty}(s,\pi)=\prod_{v\mid\infty}L(s,\pi_v) = \prod_{v\mid\infty}\prod_{j=1}^m \Gamma_v(s+\mu_{j,\pi}(v)), 
\]
where 
\begin{equation}\label{eq:gamma_v}
\Gamma_v(s)=\begin{cases}
\Gamma_{\R}(s) = \pi^{-s/2}\Gamma(s/2)&\mbox{if $F_v=\R$,}\\
\Gamma_{\mathbb{C}}(s) = 2(2\pi)^{-s}\Gamma(s)&\mbox{if $F_v=\mathbb{C}$.}
\end{cases}
\end{equation}
Luo, Rudnick, and Sarnak \cite{LRS}, along with M\"ueller  and Speh \cite{MS}, showed that there exists $\theta_m\in[0,\frac{1}{2}-\frac{1}{m^2+1}]$ such that we have the uniform bounds 
\begin{equation}
\label{eqn:GRC1}
|\alpha_{j,\pi}(\kp)|\leq \N\kp^{\theta_m},\qquad \re(\mu_{j,\pi}(v))\geq -\theta_m,
\end{equation}
and the generalized Ramanujan conjecture asserts that one may take $\theta_m=0$. 

   Let $\kq_{\pi}$ denote the conductor of $\pi$, and let $\mathbbm{1}$ denote the trivial representation (whose $L$-function is the Dedekind zeta function $\zeta_F(s)$). Let $\delta_{\pi}=1$ if $\pi=\mathbbm{1}$, and $\delta_{\pi}=0$ otherwise. Let $\tilde{\pi} \in \mathfrak{F}_m$  be the contragredient of $\pi$. The completed $L$-function
\[
\Lambda(s,\pi)=(s(1-s))^{\delta_{\pi}}(D_F^m\N\kq_{\pi})^{s/2}L_{\infty}(s,\pi)L(s,\pi)
\]
is entire of order $1$. 
There exists $W(\pi)\in\CC$ of modulus $1$ such that $\Lambda(s,\pi)$ satisfies 
\[
\Lambda(s,\pi)=W(\pi)\Lambda(1-s,\tilde\pi)=W(\pi)\overline{\Lambda(1-\bar{s},\pi)}.
\]
Since $\{\alpha_{j,\tilde{\pi}}(\kp)\} = \{\overline{\alpha_{j,\pi}(\kp)}\}$, $\N\kq_{\tilde{\pi}}=\N\kq_{\pi}$, and $\{\mu_{j,\tilde{\pi}}(v)\} = \{\overline{\mu_{j,\pi}(v)}\}$, we have that $L(s,\tilde{\pi}) = \overline{L(\overline{s},\pi)}$. If $v \mid \infty$, we define $\mathfrak{C}_v(\pi) = \mathfrak{C}_v(0,\pi)$, where 
\[
\mathfrak{C}_v(it,\pi) = \prod_{j=1}^m (|\mu_{j,\pi}(v)+it|+3)^{[F_v:\R]}.
\]
The analytic conductor is given by  $\mathfrak{C}(\pi)=\mathfrak{C}(0,\pi)$, where
\[
\mathfrak{C}(it,\pi)= D_F^m\N\kq_{\pi}\prod_{v\mid\infty} \mathfrak{C}_v(it,\pi).
\] 
Lastly, we define $a_{\pi}(\kp^k)$ by the Dirichlet series identity for any $\Real(s) > 1$
	\begin{equation}\label{eq:-L'/L}
	-\frac{L'}{L}(s,\pi) = \sum_{\kp}\sum_{k=1}^{\infty} \frac{\sum_{j=1}^{m} \alpha_{j, \pi}(\kp)^k\log \N\kp}{\N\kp^{ks}}= \sum_{\kp}\sum_{k=1}^{\infty} \frac{a_{\pi}(\kp^k)\log \N\kp}{\N\kp^{ks}}.
	\end{equation}

	\subsection{Rankin--Selberg \texorpdfstring{$L$}{L}-functions}
	\label{subsec:RS}

    Let $\pi\in\mathfrak{F}_{m}^*$ and $\pi'\in\mathfrak{F}_{m'}^*$.  Let $\delta_{\pi\times\pi'}=1$ if $\pi'=\widetilde{\pi}$ and $\delta_{\pi\times\pi'}=0$ otherwise.  For $\kp\nmid\kq_{\pi}\kq_{\pi'}$, define
\[
L(s,\pi_{\kp}\times\pi_{\kp}')=\prod_{j=1}^m \prod_{j'=1}^{m'}\frac{1}{1-\alpha_{j,\pi}(\kp)\alpha_{j',\pi'}(\kp)\N\kp^{-s}}.
\]
Jaquet, Piatetski-Shapiro, and Shalika proved the following theorem.

\begin{thm}\cite{JPS}
\label{thm:JPSS}
If $(\pi,\pi')\in\mathfrak{F}_{m}^*\times\mathfrak{F}_{m'}^*$, then there exist
\begin{enumerate}
\item complex numbers $(\alpha_{j,j',\pi\times\pi'}(\kp))_{j=1}^m{}_{j'=1}^{m'}$ for each $\kp  \mid \kq_{\pi}\kq_{\pi'}$, from which we define
	\begin{align*}
	L(s,\pi_{\kp}\times\pi_{\kp}') = \prod_{j=1}^m \prod_{j'=1}^{m'}\frac{1}{1-\alpha_{j,j',\pi\times\pi'}({\kp})\N{\kp}^{-s}}; 
	\end{align*}
\item complex numbers $(\mu_{j,j',\pi\times\pi'}(v))_{j=1}^{m}{}_{j'=1}^{m'}$ for each $v\mid\infty$, from which we define
	\begin{align*}
	L_{\infty}(s,\pi\times\pi') = \prod_{v\mid\infty}L(s,\pi_{v}\times\pi_{v}') = \prod_{v\mid\infty}\prod_{j=1}^m \prod_{j'=1}^{m'}\Gamma_v(s+\mu_{j,j',\pi\times\pi'}(v)); 
	\end{align*}
\item a conductor, an integral ideal whose norm is denoted $\N\kq_{\pi\times\pi'}=\N\kq_{\widetilde{\pi}\times\widetilde{\pi}'}$; and
\item a complex number $W(\pi\times\pi')$ of modulus $1$
\end{enumerate}
such that the Rankin--Selberg $L$-function 
\[
L(s,\pi\times\pi')=\prod_{\kp}L(s,\pi_{\kp}\times\pi_{\kp}')
\]
converges absolutely for $\Re(s)>1$, the completed $L$-functions
\begin{align}\label{eq:complete-RSL}
\Lambda(s,\pi\times\pi') &= (s(1-s))^{\delta_{\pi\times\pi'}} (D_F^{mm'}\N\kq_{\pi\times\pi'})^{\frac{s}{2}} L_{\infty}(s,\pi\times\pi')L(s,\pi\times\pi') 
\end{align}
are entire of order $1$, and 
\[
\Lambda(s,\pi\times\pi')=W(\pi\times\pi')\Lambda(1-s,\widetilde{\pi}\times\widetilde{\pi}').
\]
\end{thm}

    It follows from Theorem~\ref{thm:JPSS} that $L(s,\widetilde{\pi}\times\widetilde{\pi}')=\overline{L(\overline{s},\pi\times\pi')}.$ We can explicitly determine the numbers $\mu_{j,j',\pi\times\pi'}(v)$ at all $v\mid\infty$ (resp. $\alpha_{j,j',\pi\times\pi'}(\kp)$ at all $\kp  \mid \kq_{\pi}\kq_{\pi'}$) using the archimedean case of the local Langlands correspondence \cite[Section 3.1]{MS} (resp. \cite[Appendix]{ST}).  These descriptions and \eqref{eqn:GRC1} yield the bounds
	\begin{equation}
		\label{eqn:GRC2}
	|\alpha_{j,j',\pi\times\pi'}(\kp)|\leq  \N\kp^{\theta_m+\theta_{m'}},\qquad \mathrm{Re}(\mu_{j,j',\pi\times\pi'}(v))\geq -(\theta_m+\theta_{m'}).
	\end{equation}
	If $k \geq 1$ is an integer, then we define
	\begin{equation}
	\label{eq:a-rankin-selberg}
	a_{\pi\times\pi'}(\kp^{k})= \begin{cases}
 	a_{\pi}(\kp^{k})a_{\pi'}(\kp^{k})&\mbox{if $\kp  \nmid \kq_{\pi}\kq_{\pi'}$,}\\
 	\sum_{j=1}^m \sum_{j'=1}^{m'}\alpha_{j,j',\pi\times\pi'}(\kp)^{k}&\mbox{if $\kp  \mid \kq_{\pi}\kq_{\pi'}$,}
 \end{cases}
	\end{equation}
	We have the Dirichlet series identity
	\begin{equation}
		\label{eqn:log_deriv}
	-\frac{L'}{L}(s,\pi\times\pi')=\sum_{\kp}\sum_{k=1}^{\infty}\frac{a_{\pi\times\pi'}(\kp^{k})\log \N\kp}{\N\kp^{k s}},\qquad\mathrm{Re}(s)>1.
	\end{equation}
    
	If $v \mid \infty$, we define $\mathfrak{C}_v( \pi\times\pi') = \mathfrak{C}_v(0, \pi\times\pi')$, where
    \[
    \mathfrak{C}_v(it, \pi\times\pi') = \prod_{j=1}^m \prod_{j=1}^{m'} (|\mu_{j,j',\pi\times{\pi}'}(v)+it|+3)^{[F_v:\R]}.
    \]
    The analytic conductor  is given by $\mathfrak{C}(\pi\times\pi') = \mathfrak{C}(0, \pi\times\pi')$, where
	\begin{equation*}
	    \begin{aligned}
	        \mathfrak{C}(it, \pi\times\pi')=D_F^{mm'}\N\kq_{\pi\times{\pi}'}\prod_{v\mid\infty} \mathfrak{C}_v(it, \pi\times\pi').
	    \end{aligned}
	\end{equation*}
    We have the inequality    \begin{equation}\label{eq:conductor-ineq}
        \mathfrak{C}(it, \pi\times \pi') \le  \mathfrak{C}(\pi)^{m'}\mathfrak{C}(\pi')^m(|t|+3)^{mm'[F:\Q]},
    \end{equation}
    as well as $\mathfrak{C}(\pi \times \pi') \le \mathfrak{C}(\pi)^{m'}\mathfrak{C}(\pi')^{m}.$ 
    We provide a detailed proof of \eqref{eq:conductor-ineq} in the appendix of this article, following \cite[Theorem A.2]{Humphries-Brumley} and the correction noted in \cite[P. 9]{Ramakrishnan-Yang}. 

	\section{Preliminary computations and useful bounds}\label{sec:gen-strategy}

   We begin by modifying the family of metrics defined in \eqref{def:distance}. Let $\pi \in \cup_{m=1}^{\infty}\mathfrak{F}^*_m$, $\delta \in \{\pm1\}$, and $\gamma \in \mathbb{R}$. We define $M$ to be the set of functions $\kn \mapsto \delta a_{\pi}(\kn) \mathrm{N}\mathfrak{n}^{i\gamma}$, where $a_{\pi}(\mathfrak{n})$ denote the Dirichlet coefficient of $-(L'/L)(s, \pi)$ at $\kn$.
    Let $m_1, m_2 > 0$ be integers and $\pi_1 \in \mathfrak{F}_{m_1}^*$, $\pi_2 \in \mathfrak{F}_{m_2}^*$. For $\gamma_1, \gamma_2 \in \mathbb{R}$ and $\delta_1, \delta_2 \in \{\pm 1\}$, we define $\mathbb{D}_{\sigma}(\cdot, \cdot)$ on $M$ as
\begin{equation} \label{eq:gen-distance}
\begin{aligned}
&\mathbb{D}_{\sigma}(\delta_1 a_{\pi_1}(\mathfrak{n}) \mathrm{N}\mathfrak{n}^{i\gamma_1}, \delta_2 a_{\pi_2}(\mathfrak{n}) \mathrm{N}\mathfrak{n}^{i\gamma_2})\\ 
&\qquad := \sqrt{\frac{1}{2} \sum_{\mathfrak{p}} \sum_{k=1}^{\infty} \frac{|\delta_1 a_{\pi_1}(\mathfrak{p}^k)\mathrm{N}\mathfrak{p}^{ik\gamma_1} - \delta_2 a_{\pi_2}(\mathfrak{p}^k)\mathrm{N}\mathfrak{p}^{ik\gamma_2}|^2 \log \mathrm{N}\mathfrak{p}}{\mathrm{N}\mathfrak{p}^{k\sigma}} },
\end{aligned}
\end{equation}
where $\sigma > 1$ is a parameter to be optimized later. We note that $\delta_1$ and $\delta_2$ are not limited to $\{\pm1\}$, but this choice suffices for our purposes and simplifies the computations.
To simplify notation, for $j \in \{1, 2, 3\}$, we set $x_j(\mathfrak{n}) := \delta_j a_{\pi_j}(\mathfrak{n}) \mathrm{N}\mathfrak{n}^{i\gamma_j}$. The quantity in \eqref{eq:gen-distance} defines a family of metrics on $M$  as it satisfies the following properties:
\begin{enumerate}
    \item \textbf{Nonnegativity}: ``$\mathbb{D}_{\sigma}(x_1(\mathfrak{n}), x_2(\mathfrak{n})) \ge 0$, with equality if and only if $x_1(\mathfrak{n}) = x_2(\mathfrak{n})$.'' The non-negativity follows from the definition in \eqref{eq:gen-distance}. For the second statement, the ``only if'' direction also follows immediately from the definition. For the ``if'' direction, suppose that $\mathbb{D}_{\sigma}(x_1(\mathfrak{n}), x_2(\mathfrak{n})) = 0$. This implies that 
    \[
    x_1(\mathfrak{p}^k) = x_2(\mathfrak{p}^k) \text{ for all prime powers } \mathfrak{p}^k.
    \]
    Since $a_{\pi_1}(\kn)$ and $a_{\pi_2}(\kn)$, and hence $x_1(\kn)$ and $x_2(\kn)$, are supported on prime powers, then $x_1(\mathfrak{n}) = x_2(\mathfrak{n}) = 0$ for all $\mathfrak{n}$ that are not prime powers. It follows that $x_1(\mathfrak{n}) = x_2(\mathfrak{n})$ for all $\mathfrak{n}$ as desired. 
    
    \item \textbf{Symmetry:} By \eqref{eq:gen-distance}, we have $\mathbb{D}_{\sigma}(x_1(\mathfrak{n}), x_2(\mathfrak{n})) = \mathbb{D}_{\sigma}(x_2(\mathfrak{n}), x_1(\mathfrak{n})).$
    
    \item \textbf{Triangle inequality:} The inequality
    \begin{equation} \label{eq:triangle-D-sigma}
    \mathbb{D}_{\sigma}(x_1(\mathfrak{n}), x_2(\mathfrak{n})) \le \mathbb{D}_{\sigma}(x_1(\mathfrak{n}), x_3(\mathfrak{n})) + \mathbb{D}_{\sigma}(x_3(\mathfrak{n}), x_2(\mathfrak{n}))
    \end{equation}
    follows from \eqref{eq:gen-distance} and Minkowski’s inequality.
\end{enumerate}

In Section~\ref{sec:proof-level1}, we will apply the new family of metrics in \eqref{eq:gen-distance} to establish zero-free regions for  $L(s, \pi)$ and $L(s, \pi \times \pi')$, by carefully choosing $\pi_j, \gamma_j$, and $\delta_j$ for $j \in \{1,2,3\}$, and using an appropriate triangle inequality of the form \eqref{eq:triangle-D-sigma}. Before proceeding, we carry out preliminary computations and derive useful bounds that will be used throughout the rest of the paper. 

\subsection{Preliminary computations} 
    Consider the right-hand side of \eqref{eq:gen-distance}. Direct computation gives   
	\begin{multline*}
			|\delta_1 a_{\pi_1}(\kp^k)\N\kp^{ik\gamma_1}- \delta_2 a_{\pi_2}(\kp^k)\N\kp^{ik\gamma_2}|^2\\ = |a_{\pi_1}(\kp^k)|^2+|a_{\pi_2}(\kp^k)|^2 -2\delta_1\delta_2 \Real(a_{\pi_1}(\kp^k)\overline{a_{\pi_2}(\kp^k)}\N\kp^{ik(\gamma_1-\gamma_2)}).
	\end{multline*}
	 It follows that 
    \begin{equation}\label{eq:gen-distance-1}
	\begin{aligned}  
         &\mathbb{D}_{\sigma}(\delta_1 a_{\pi_1}(\kn)\N\kn^{i\gamma_1},\delta_2 a_{\pi_2}(\kn)\N\kn^{i\gamma_2})^2 \\
         &\quad = \frac{1}{2} \sum_{\kp} \sum_{k=1}^{\infty} \frac{|a_{\pi_1}(\kp^k)|^2\log \N\kp} {\N\kp^{k\sigma}} 
         + \frac{1}{2} \sum_{\kp} \sum_{k=1}^{\infty} \frac{|a_{\pi_2}(\kp^k)|^2\log \N\kp} {\N\kp^{k\sigma}} \\ 
         &\qquad - \delta_1\delta_2 \Real\Big(\sum_{\kp} \sum_{k=1}^{\infty}\frac{a_{\pi_1}(\kp^k)\overline{a_{\pi_2}(\kp^k)}\log \N\kp} {\N\kp^{k\sigma}\N\kp^{ik(\gamma_2-\gamma_1)}}\Big).
        \end{aligned}
	\end{equation}
    
     In the proofs of Theorems~\ref{thm1:ZFR-cuspforms}--\ref{thm2:ZFR-rankin-selberg},  choices of $\pi_1$ and $\pi_2$  
    will be chosen from $\{\pi, \pi',\tilde{\pi},\tilde{\pi}'\}$. If $\N\kq_{\pi} = \N\kq_{\pi'} =1$,  by \eqref{eq:a-rankin-selberg} and \eqref{eqn:log_deriv}, we have 
    \begin{multline}
        \label{eq:gen-distance-2}
        \mathbb{D}_{\sigma}(\delta_1 a_{\pi_1}(\kn)\N\kn^{i\gamma_1},\delta_2 a_{\pi_2}(\kn)\N\kn^{i\gamma_2})^2 = -\frac{1}{2}\frac{L'}{L}(\sigma, \pi_1 \times \tilde{\pi}_1) -\frac{1}{2}\frac{L'}{L}(\sigma, \pi_2 \times \tilde{\pi}_2)\\
        +\delta_1\delta_2 \Real\Big(\frac{L'}{L}(\sigma+(\gamma_2-\gamma_1)i, \pi_1 \times \tilde{\pi}_2)\Big).
      \end{multline}
      If $\N\kq_{\pi}>1$ or $\N\kq_{\pi'} >1$, the first term on the right-hand side  of \eqref{eq:gen-distance-1} equals
	\begin{equation*}
		\begin{aligned}[b]
		   \frac{1}{2} \sum_{\kp} \sum_{k=1}^{\infty} \frac{|a_{\pi_1}(\kp^k)|^2\log \N\kp} {\N\kp^{k\sigma}} &= \frac{1}{2} \sum_{\kp  \nmid \kq_{\pi_1}} \sum_{k=1}^{\infty} \frac{|a_{\pi_1}(\kp^k)|^2\log \N\kp} {\N\kp^{k\sigma}} + 
			\frac{1}{2} \sum_{\kp  \mid \kq_{\pi_1}} \sum_{k=1}^{\infty} \frac{a_{\pi_1 \times \tilde{\pi}_1}(\kp^k)\log \N\kp} {\N\kp^{k\sigma}}\\
            &+ \frac{1}{2} \sum_{\kp  \mid \kq_{\pi_1}} \sum_{k=1}^{\infty} \frac{|a_{\pi_1}(\kp^k)|^2\log \N\kp} {\N\kp^{k\sigma}} 
			- \frac{1}{2} \sum_{\kp  \mid \kq_{\pi_1}} \sum_{k=1}^{\infty} \frac{a_{\pi_1 \times \tilde{\pi}_1}(\kp^k)\log \N\kp} {\N\kp^{k\sigma}},
		\end{aligned}
	\end{equation*}
    which, by \eqref{eq:a-rankin-selberg} and \eqref{eqn:log_deriv}, equals
    \begin{equation}
    \label{eq:gen-distance-term1-starter}
     -\frac{1}{2}\frac{L'}{L}(\sigma, \pi_1 \times \tilde{\pi}_1) +   
			\frac{1}{2} \Big(\sum_{\kp  \mid \kq_{\pi_1}} \sum_{k=1}^{\infty} \frac{|a_{\pi_1}(\kp^k)|^2\log \N\kp} {\N\kp^{k\sigma}} 
			- \sum_{\kp  \mid \kq_{\pi_1}} \sum_{k=1}^{\infty} \frac{a_{\pi_1 \times \tilde{\pi}_1}(\kp^k)\log \N\kp} {\N\kp^{k\sigma}}\Big).
    \end{equation}
	For convenience, we define 
	\begin{equation}\label{eq:def-E}
		E(\sigma+ i\gamma, \pi_1 \times \pi_2) :=  \sum_{\kp  \mid \kq_{\pi_1}\kq_{\pi_2}} \sum_{k=1}^{\infty} \frac{a_{\pi_1}(\kp^k)a_{\pi_2}(\kp^k)\log \N\kp}{\N\kp^{k(\sigma+i\gamma)}} 
		- \sum_{\kp  \mid \kq_{\pi_1}\kq_{\pi_2}} \sum_{k=1}^{\infty} \frac{a_{\pi_1 \times \pi_2}(\kp^k)\log \N\kp} {\N\kp^{k(\sigma+i\gamma)}}.
	\end{equation}
        Then \eqref{eq:gen-distance-term1-starter} can be written as
	\begin{equation*}
		 -\frac{1}{2}\frac{L'}{L}(\sigma, \pi_1 \times \tilde{\pi}_1) +   
		\frac{1}{2} E(\sigma, \pi_1 \times \widetilde{\pi}_1) .
	\end{equation*}
   The second and third terms on the right-hand side of \eqref{eq:gen-distance-1} can be rewritten in the same manner as described above. Consequently, it follows from \eqref{eq:gen-distance-1} that
    	\begin{equation}
		\label{eq:gen-distance-3}
		\begin{aligned}[b]
			&\mathbb{D}_{\sigma}(\delta_1 a_{\pi_1}(\kn)\N\kn^{i\gamma_1},\delta_2 a_{\pi_2}(\kn)\N\kn^{i\gamma_2})^2 \\
			&\quad = -\frac{1}{2}\frac{L'}{L}(\sigma, \pi_1 \times \tilde{\pi}_1)-\frac{1}{2} \frac{L'}{L}(\sigma, \pi_2 \times \tilde{\pi}_2)   +   \delta_1\delta_2\Real\Big(\frac{L'}{L}(\sigma + (\gamma_2-\gamma_1)i, \pi_1 \times \tilde{\pi}_2) \Big)\\
			&\qquad +
			\frac{1}{2}E(\sigma, \pi_1 \times \tilde{\pi}_1) +  \frac{1}{2}E(\sigma, \pi_2 \times \tilde{\pi}_2) -\delta_1\delta_2 \Real(E(\sigma+(\gamma_2-\gamma_1)i, \pi_1 \times \tilde{\pi}_2)).
		\end{aligned}
	\end{equation}

	\subsection{Useful bounds} 
	
	In the proofs of Theorems~\ref{thm1:ZFR-cuspforms}--\ref{thm2:ZFR-rankin-selberg}, we will need to estimate the right-hand side of \eqref{eq:gen-distance-2} in the case where $\N\kq_{\pi} = \N\kq_{\pi'} =1$ and the right-hand side of \eqref{eq:gen-distance-3} in the case where $\N\kq_{\pi}>1$ or $\N\kq_{\pi'}>1$. To proceed, we require the following lemmas.
    \begin{lem}
    \label{lem:stirling}
        Let $\pi_1 \in \mathfrak{F}_{m_1}^*$ and $\pi_2 \in \mathfrak{F}_{m_2}^*$, and   let $\rho$ denote the nontrivial zeros of $L(s,\pi_1 \times \pi_2)$. Let $\delta_{\pi_1\times \pi_2}$ be defined as in Section~\ref{subsec:RS}, and let $\gamma_{\mathbb{Q}}$ be the Euler--Mascheroni constant. 
        Let $\sigma, t \in \mathbb{R}$, and define $\Cl[abcon]{hadamard} = \log{\pi}+\gamma_{\mathbb{Q}} > 1.721$.  If $1<\sigma<2$ with  $s=\sigma+it$, then
        \begin{equation*}
            -\Real\Big(\frac{L'}{L}(s, \pi_1 \times \pi_2)\Big) \le  \Real\Big(\frac{\delta_{\pi_1\times\pi_2}}{s-1}\Big) - \sum_{\rho}\Real\Big( \frac{1}{s-\rho}\Big)
            +\Big(\delta -\frac{\Cr{hadamard}}{2}\Big) + \frac{1}{2}\log\mathfrak{C}(it, \pi_1 \times \pi_2).
        \end{equation*}
    \end{lem}
    \begin{proof} Since $\Lambda(s, \pi_1 \times \pi_2)$ as defined in \eqref{eq:complete-RSL} is entire, it admits a Hadamard factorization
    \begin{equation}\label{eq:hadamard}
        \Lambda(s, \pi_1 \times \pi_2)= e^{a+bs} \prod_{\rho}\Big(1-\frac{s}{\rho}\Big)e^{s/\rho},
    \end{equation}
    where $a=a(\pi_1\times\pi_2)$ and $b = b(\pi_1\times \pi_2)$ are constants, and $\rho$ ranges over zeros of  $\Lambda(s, \pi_1 \times \pi_2)$.
    Using the definition \eqref{eq:complete-RSL}, we  take the real parts of the logarithmic derivatives of both sides of \eqref{eq:hadamard} and apply the fact that $\Real(b) = -\sum_{\rho}\Real(1/\rho)$, which follows from \cite[Theorem 5.7]{IK}. It follows that
        \begin{multline}\label{eq:hadamard-prod}
            -\Re\Big(\frac{L'}{L}(s, \pi_1 \times \pi_2)\Big) = \Re\Big(\frac{\delta_{\pi_1\times \pi_2}}{s} +\frac{\delta_{\pi_1\times \pi_2}}{s-1}\Big) - \sum_{\rho} \Re\Big(\frac{1}{s-\rho}\Big)\\
            + \frac{1}{2} \log(D_F^{m_1m_2}\N\kq_{\pi_1\times \pi_2}) + \Re\Big(\frac{L_{\infty}'}{L_{\infty}}(s, \pi_1 \times \pi_2)\Big),
        \end{multline}
        leaving us to bound the last two terms.
        
        By the definition of $L_{\infty}(s, \pi_1 \times \pi_2)$ in Theorem~\ref{thm:JPSS}, we have that 
        \begin{equation}\label{eq:gamma-estimate}
        \Real\Big(\frac{L_{\infty}'}{L_{\infty}}(s, \pi_1 \times \pi_2) \Big)  = \Real\Big(\sum_{v\mid\infty}\sum_{j_1=1}^{m_1} \sum_{j_2=1}^{m_2}\frac{\Gamma'_v}{\Gamma_v}(s+\mu_{j_1,j_2,\pi_1\times\pi_2}(v))\Big).
        \end{equation} 
       Recall the definition of $\Gamma_v(s)$ in \eqref{eq:gamma_v}. By  \cite[Lemma 4.1]{HIJT} and the fact that $\Gamma_{\mathbb{C}}(s) = \Gamma_{\R}(s)\Gamma_{\R}(s+1)$, there is a constant $\Cr{hadamard} = \log{\pi}+ \gamma_{\mathbb{Q}}$ such that 
        \begin{equation}\label{eq:HIJT} \Real\Big(\frac{\Gamma_v'}{\Gamma_v}(s)\Big) \le [F_v:\mathbb{R}]\Big(-\frac{\Cr{hadamard}}{2} +\frac{1}{2}\log(|s+1|)\Big). 
        \end{equation}
    Denoting $\mu_{j_1, j_2, \pi_1 \times \pi_2}(v)$ with $\mu_{j_1,j_2}(v)$, we have $\log(|s+\mu_{j_1,j_2}+1|) \le  \log(|\mu_{j_1,j_2}+it|+3)$, using $\sigma <2$. Thus it follows from \eqref{eq:gamma-estimate} and \eqref{eq:HIJT} that
    \begin{equation*}
    \begin{aligned}
    \Real\Big(\frac{L_{\infty}'}{L_{\infty}}(s, \pi_1 \times \pi_2) \Big)
         \le -\frac{\Cr{hadamard}mm'[F:\mathbb{Q}]}{2}+ \frac{1}{2}\Real\Big(\sum_{v\mid\infty}\sum_{j=1}^m \sum_{j'=1}^{m'}[F_v:\mathbb{R}]\log(|\mu_{j_1,j_2}+it|+3)\Big).
    \end{aligned}
    \end{equation*}
  By using the above inequality in \eqref{eq:hadamard-prod} and applying the definition of $\mathfrak{C}(it, \pi_1 \times \pi_2)$ as defined in Section~\ref{subsec:RS}, we obtain
   \begin{equation}
   \begin{aligned}
            -\Real\Big(\frac{L'}{L}(s, \pi_1 \times \pi_2)\Big) &\le  \Real\Big(\frac{\delta_{\pi_1\times \pi_2}}{s}+\frac{\delta_{\pi_1\times \pi_2}}{s-1}\Big) - \sum_{\rho} \Real\Big( \frac{1}{s-\rho}\Big)
            -\frac{\Cr{hadamard}mm'[F:\mathbb{Q}]}{2}\\ &\quad
            +  \frac{1}{2}\log\mathfrak{C}(it, \pi_1 \times \pi_2).
        \end{aligned}
        \end{equation}
    The lemma then follows from the fact that $\sigma > 1$ and $mm'[F:\mathbb{Q}] \ge 1$.
    \end{proof}
    
    We note that the contributions from zeros $\rho$ of $L(s, \pi_1 \times \pi_2)$ in Lemma~\ref{lem:stirling} can be handled by observing that for $\sigma > 1$ and any nontrivial zero $\rho = \beta+i\gamma$, we have 
	\begin{equation}
		\label{eq:positivity}
		\Real\Big(\frac{1}{\sigma+it -\rho}\Big) = \frac{\sigma -\beta}{(\sigma-\beta)^2+(t - \gamma)^2} \ge 0. 
	\end{equation}
    Now, we let $\omega(\kn)$ denote the number of distinct prime ideals dividing the ideal $\kn$, namely,
    \[
     \omega(\kn) := | \{ \kp :  \kp \mid \kn\} |.
    \]
    The next lemma provides an estimate for $E(\sigma+i\gamma, \pi_1 \times \pi_2)$, which arises only when $\N\kq_{\pi}>1$ or $\N\kq_{\pi'}>1$, given that we take $\pi_1, \pi_2 \in \{\pi,\pi',\tilde{\pi},\tilde{\pi}'\}$. 
    
    \begin{lem}\label{lem:general-error}
		Recall the definition of $E(\sigma+it, \pi_1 \times \pi_2)$ in \eqref{eq:def-E}. Let $\pi_1 \in \mathfrak{F}_{m_1}^*$ and $\pi_2 \in \mathfrak{F}_{m_2}^*$.
        For $i \in\{1,2\}$, let $\theta_{m_i} \in[0,\frac{1}{2}-\frac{1}{m_i^2+1}]$ be the best bound towards the generalized Ramanujan conjecture for $\pi_i$. Then we have 
		\begin{equation*}
			\begin{aligned}
				|E(\sigma+ i\gamma, \pi_1\times \pi_2)|
				\ll \frac{m_1m_2}{1-\theta_{m_1}-\theta_{m_2}}\omega(\kq_{\pi_1}\kq_{\pi_2}).
			\end{aligned}
		\end{equation*}
	\end{lem}
	
	\begin{proof}
		For $i \in \{1,2\}$, we first note that 
		\begin{equation*}\label{eq:localroot}
			|a_{\pi_i}(\kp^k)| = \Big|\sum_{j=1}^{m_i} \alpha_{j,\pi_i}(\kp)^k \Big| \le \sum_{j=1}^{m_i} \N\kp^{k\theta_{m_i}} = m_i\N\kp^{k\theta_{m_i}}.
		\end{equation*}
		By \eqref{eqn:GRC1}, \eqref{eq:-L'/L},  \eqref{eqn:GRC2}, and \eqref{eq:a-rankin-selberg}, for $\kp  \mid \kq_{\pi_1}\kq_{\pi_2}$, we have
		\begin{equation*}\label{eq:localroot2}
        \begin{aligned}
        & |a_{\pi_1}(\kp^k)||a_{\pi_2}(\kp^k)| \le m_1m_2\N\kp^{k(\theta_{m_1}+\theta_{m_2})}; \\
        &|a_{\pi_1 \times \pi_2}(\kp^k)| =\Big|\sum_{j=1}^{m_1}\sum_{j=1}^{m_2} \alpha_{j_1,j_2, \pi_1 \times \pi_2}(\kp)^k\Big|  \le m_1m_2 \N\kp^{k(\theta_{m_1}+\theta_{m_2})}.
         \end{aligned}
		\end{equation*}
		Since $\sigma>1$, we have 
		\begin{equation*}\label{eq:error}
			\begin{aligned}
				|E(\sigma+ i\gamma, \pi_1 \times \pi_2)|
				&\ll  \sum_{\kp  \mid \kq_{\pi_1}\kq_{\pi_2}} \sum_{k=1}^{\infty} \frac{m_1m_2
                \N\kp^{k(\theta_{m_1}+\theta_{m_2})}\log \N\kp}{\N\kp^{k}} \\
                &=  m_1m_2\sum_{\kp  \mid \kq_{\pi_1}\kq_{\pi_2}}\frac{\log \N\kp}{\N\kp^{1-\theta_{m_1}-\theta_{m_2}}-1}.
			\end{aligned}
		\end{equation*}
        The lemma then follows from the observation that
	\begin{equation*}
			\begin{aligned}
		(1-\theta_{m_1}-\theta_{m_2})\log \N\kp =\log(\N\kp^{1-\theta_{m_1}-\theta_{m_2}})
        \leq \sum_{j=1}^{\infty}\frac{(\log(\N\kp^{1-\theta_{m_1} -\theta_{m_2}}))^{j}}{j!},
		\end{aligned} 
		\end{equation*} 
        which equals $\N\kp^{1-\theta_{m_1}-\theta_{m_2}}-1$ as desired. 
	\end{proof}

	\section{Proofs of Theorems~\ref{thm1:ZFR-cuspforms}--\ref{thm2:ZFR-rankin-selberg} when \texorpdfstring{$\N\kq_{\pi}=\N\kq_{\pi'} = 1$}{N1}}\label{sec:proof-level1}

    In this section, we establish Theorems~\ref{thm1:ZFR-cuspforms}--\ref{thm2:ZFR-rankin-selberg} under the assumption that $\N\kq_{\pi} = \N\kq_{\pi'} = 1$. 

    	\subsection{Proof of Theorem~\ref{thm1:ZFR-cuspforms}}\label{subsec:proof-level1-thm1}
	
	Let $\pi\in\mathfrak{F}_m^*$ and $\rho = \beta +i\gamma$ be a nontrivial zero of $L(s, \pi)$. 
	Assume that for any constant $\Cl[abcon]{zfr0} > 0$, we have
	\begin{equation}\label{eq:assumption-beta1}
		\begin{aligned}
			\beta
			\ge 1- \frac{1}{\Cr{zfr0}\mathcal{L}_1}, \qquad \mathcal{L}_1 = (2m+3)\log(\mathfrak{C}(\pi) (3+|\gamma|)^{m[F:\Q]}).
		\end{aligned}
	\end{equation}
	We will show that \eqref{eq:assumption-beta1} does not hold for sufficiently large $\Cr{zfr0}$. 
    
    Consider the triangle inequality
	\begin{equation*}
		\begin{aligned}
			\mathbb{D}_{\sigma}(a_{\pi}(\kn)\N\kn^{-i\gamma}, \overline{a_{\pi}(\kn)}\N\kn^{i\gamma}) &\le 
			\mathbb{D}_{\sigma}( a_{\pi}(\kn)\N\kn^{-i\gamma}, -1) + \mathbb{D}_{\sigma}(-1, \overline{a_{\pi}(\kn)}\N\kn^{i\gamma})\\
			&= 2\mathbb{D}_{\sigma}( a_{\pi}(\kn)\N\kn^{-i\gamma},-1).
		\end{aligned}
	\end{equation*}
	By squaring both sides of the above inequality, it follows that
	\begin{equation*}
		\mathbb{D}_{\sigma}(a_{\pi}(\kn)\N\kn^{-i\gamma}, \overline{a_{\pi}(\kn)}\N\kn^{i\gamma})^2 \le  4\mathbb{D}_{\sigma}( a_{\pi}(\kn)\N\kn^{-i\gamma},-1)^2.
	\end{equation*}
	Applying the computation in \eqref{eq:gen-distance-2} to the above inequality, we then have
	\begin{equation}\label{eq:thm1-eq1}
		0 \le -\frac{L'}{L}(\sigma, \pi\times \tilde{\pi})-2\frac{\zeta_F'}{\zeta_F}(\sigma) -\Real\Big(\frac{L'}{L}(\sigma+2i\gamma, \pi \times \pi)\Big)-4\Real\Big(\frac{L'}{L}(\sigma+i\gamma, \pi)\Big).
	\end{equation}
	
	For the first term in \eqref{eq:thm1-eq1}, we apply Lemma~\ref{lem:stirling} and use \eqref{eq:positivity} to drop the contribution of all nontrivial zeros, yielding 
	\begin{equation}
		\label{eq:cuspform-term1-2}
		\begin{aligned}
			-\frac{L'}{L}(\sigma, \pi \times \tilde{\pi}) 
			&\le \frac{1}{\sigma-1}
			+\Big(1 -\frac{\Cr{hadamard}}{2}\Big)+\frac{1}{2}\log \mathfrak{C}(\pi \times \tilde{\pi}).
		\end{aligned}
	\end{equation}	
	Similarly, for the second term of \eqref{eq:thm1-eq1}, we have 
	\begin{equation}
		\label{eq:cuspform-term2}
		-2\frac{\zeta_F'}{\zeta_F}(\sigma) \le  \frac{2}{\sigma-1}+(2-\Cr{hadamard})
		+ \log  \mathfrak{C}( \mathbbm{1}).
	\end{equation}
    For the last term of \eqref{eq:thm1-eq1}, we apply Lemma~\ref{lem:stirling} and use \eqref{eq:positivity}
	to remove the contribution of all nontrivial zeros except for $\rho=\beta+i\gamma$. It follows that
	\begin{equation}
		\label{eq:cuspform-term4}
		-4\Real\Big(\frac{L'}{L}(\sigma+i\gamma, \pi)\Big)\le \frac{4\delta_{\pi = \mathbbm{1}}(\sigma-1)}{(\sigma-1)^2+\gamma^2} - \frac{4}{\sigma-\beta}+4\delta_{\pi = \mathbbm{1}}-2\Cr{hadamard}
		+ 2\log \mathfrak{C}(i\gamma, \pi).
	\end{equation}
	We then consider the following two cases.
	\subsubsection{Case 1: \texorpdfstring{$\pi$}{pi} is not self-dual}\label{subsubsec:case1-thm1}
	Then 
	$L(s, \pi \times \pi)$ is entire. For the third term of \eqref{eq:thm1-eq1}, by applying Lemma~\ref{lem:stirling} and \eqref{eq:positivity}, we  obtain
	\begin{equation}
		\label{eq:cuspform-term3}
		-\Real\Big(\frac{L'}{L}(\sigma+2i\gamma, \pi \times \pi)\Big) \le -\frac{\Cr{hadamard}}{2} + \frac{1}{2}\log \mathfrak{C}(2i\gamma, \pi \times \pi).
	\end{equation}
	Inserting \eqref{eq:cuspform-term1-2}--\eqref{eq:cuspform-term3} in \eqref{eq:thm1-eq1} and using Lemma~\ref{lem:brumley} and the fact that $3-4\Cr{hadamard} < 0$, we have 
	\begin{equation}\label{eq:thm1-3-4-1}
		\begin{aligned}
			0 
			&\le \frac{3}{\sigma-1}- \frac{4}{\sigma-\beta} + \mathcal{L}_1.
		\end{aligned}
	\end{equation}
	Let $\Cl[abcon]{sigma-cusp} > 0$ denote a sufficiently large constant, and choose $\sigma = 1 + {1}/({\Cr{sigma-cusp}\mathcal{L}_1})$. 
	Using \eqref{eq:assumption-beta1} and substituting our choice of $\sigma$ into \eqref{eq:thm1-3-4-1}, it remains to determine the smallest $\Cr{zfr0} > 0$ for which 
	\begin{equation}\label{eq:thm1.1-finaleq}
		0 \le 3\Cr{sigma-cusp}\mathcal{L}_1-\frac{4\Cr{zfr0}\Cr{sigma-cusp}}{\Cr{zfr0}+\Cr{sigma-cusp}}\mathcal{L}_1+\mathcal{L}_1
	\end{equation}
	fails for some $\Cr{sigma-cusp} > 0$. Since \eqref{eq:thm1.1-finaleq} fails only when $\Cr{sigma-cusp} >1$, we then contradict \eqref{eq:thm1.1-finaleq} once  taking
	\[
	\Cr{zfr0} >  \min_{\Cr{sigma-cusp} > 1} \frac{3\Cr{sigma-cusp}^2+\Cr{sigma-cusp}}{\Cr{sigma-cusp}-1} =  13.9282\ldots
	\]
	
 \subsubsection{Case 2: \texorpdfstring{$\pi$}{pi} is self-dual}\label{subsubsec:Thm1-case2} Then $L(s, \pi \times \pi)$ has a simple pole at $s=1$. Using Lemma~\ref{lem:stirling} and \eqref{eq:positivity}, we obtain the following bound for the third term of \eqref{eq:thm1-eq1}:
\begin{equation}
	\label{eq:cuspform-term3-1}
	-\Real\Big(\frac{L'}{L}(\sigma+2i\gamma, \pi \times \pi)\Big) \le \frac{\sigma-1}{(\sigma-1)^2 +4\gamma^2}
	+\Big(1 -\frac{\Cr{hadamard}}{2}\Big)+ \frac{1}{2}\log \mathfrak{C}(2i\gamma, \pi \times \pi).
\end{equation}
Inserting \eqref{eq:cuspform-term1-2}--\eqref{eq:cuspform-term4} and \eqref{eq:cuspform-term3-1} in \eqref{eq:thm1-eq1} and using Lemma~\ref{lem:brumley}, we have
\begin{equation}
	\label{eq:cusp-case2}
	\begin{aligned}
		0 
		&\le \frac{3}{\sigma-1}+ \frac{4(\sigma-1)}{(\sigma-1)^2+\gamma^2} + \frac{\sigma-1}{(\sigma-1)^2 +4\gamma^2} - \frac{4}{\sigma-\beta}+ (8-4\Cr{hadamard})+\mathcal{L}_1. 
	\end{aligned}
\end{equation}
    Let $\Cl[abcon]{sigma-cusp-3}>0$ denote a sufficiently large constant.
    We consider the following three scenarios.
    \begin{enumerate}[ leftmargin=1cm]
        \item\label{thm1-case3.1}  Suppose that
        $|\gamma| \ge {4}/({\Cr{sigma-cusp-3}\mathcal{L}_1})$. 
        Using the definition of $\mathcal{L}_1$ in \eqref{eq:assumption-beta1}, the last two terms of \eqref{eq:cusp-case2} can be bounded above by
    \[
    \Big(\frac{8-4\Cr{hadamard}}{4\log 3} + 1\Big) \mathcal{L}_1 \le 1.2531\mathcal{L}_1 := k\mathcal{L}_1.  
    \]
    Choosing $\sigma = 1 + {1}/({\Cr{sigma-cusp-3}\mathcal{L}_1})$, we have $\gamma \ge 4(\sigma-1)$. It follows from \eqref{eq:cusp-case2} that
    \begin{equation}\label{eq:pi-1}
        0 \le \frac{3592}{1105(\sigma-1)} -\frac{4}{\sigma-\beta}
        +  k\mathcal{L}_1.
    \end{equation}
    Using \eqref{eq:assumption-beta1} and our choice of $\sigma$, we contradict \eqref{eq:pi-1} when 
    \begin{equation}\label{eq:minizer-cupsform-2}
    \Cr{zfr0} >  \min_{\Cr{sigma-cusp-3} > 1105k/828} \frac{3592\Cr{sigma-cusp-3}^2+1105k\Cr{sigma-cusp-3}}{828\Cr{sigma-cusp-3}-1105k} =  32.2770\ldots,
    \end{equation}
    where the minimum is attained when $\Cr{sigma-cusp-3} = 3.5273\ldots$
    
    \item\label{thm1-case3.2} Suppose that $0 < |\gamma| < {4}/({\Cr{sigma-cusp-3}\mathcal{L}_1})$. Since $\gamma \neq 0$, then both $\rho$ and $\overline{\rho}$ are  distinct zeros of $L(s,\pi)$. Then by Lemma~\ref{lem:stirling} and  \eqref{eq:positivity}, we have
    \begin{equation}\label{eq:thm1-case3.2}
        -\frac{L'}{L}(\sigma, \pi) \le \frac{1}{\sigma-1} - \frac{2(\sigma-\beta)}{(\sigma-\beta)^2+\gamma^2} + \Big(1-\frac{\Cr{hadamard}}{2}\Big)+ \frac{1}{2}\log \mathfrak{C}(\pi).
    \end{equation}
        By the definition of $\mathcal{L}_1$, the last two terms of \eqref{eq:thm1-case3.2} can be bounded above by
        \[
        \Big(\frac{2-\Cr{hadamard}}{8\log 3} + \frac{1}{8}\Big) \mathcal{L}_1 \le 0.1567\mathcal{L}_1 := k'\mathcal{L}_1.  
        \]
         We observe that  $-(L'/L)(\sigma, \pi) \ge 0$ when $\sigma > 1$. Choosing $\sigma = 1 + {8}/({\Cr{sigma-cusp-3}\mathcal{L}_1})$, it follows that  $|\gamma| < \frac{1}{2}|\sigma-\beta|$. It then follows from \eqref{eq:thm1-case3.2} that
        \begin{equation}\label{eq:pi-2}
         0 \le \frac{1}{\sigma-1} -\frac{8}{5(\sigma-\beta)}
        +  k'\mathcal{L}_1.   
        \end{equation}
        Using \eqref{eq:assumption-beta1}, our choice of $\sigma$, and the value of $\Cr{sigma-cusp-3}$ from \eqref{eq:minizer-cupsform-2}, we contradict \eqref{eq:pi-2} when 
        \[
        \Cr{zfr0} > \frac{5\Cr{sigma-cusp-3}^2+40k'\Cr{sigma-cusp-3}}{24\Cr{sigma-cusp-3}-320k'} =  2.4431\ldots
        \]
        Hence, taking $\Cr{zfr0} = 33$, we have proved a zero-free region for Cases \eqref{thm1-case3.1} and \eqref{thm1-case3.2}.
    
        \item 
        
        Lastly, we handle potential real zeros of $L(s, \pi)$.
        Let $N$ be the number of all real zeros, counted with multiplicity, in the range $1- 1/(33\mathcal{L}_1) \le \beta \le 1$. Consider $-(L'/L)(\sigma,\pi)$ and use Lemma~\ref{lem:stirling}, \eqref{eq:positivity}, and the fact that $-(L'/L)(\sigma,\pi) \ge 0$ when $\sigma >1$. We have
        \[
        0 \le \frac{1}{\sigma-1}-\frac{N}{\sigma-\beta}+k'\mathcal{L}_1.
        \]
        Using the same choice of $\sigma$ as in Case \eqref{thm1-case3.2}, we obtain
        \[
        0 \le \frac{\Cr{sigma-cusp-3}}{8}-\frac{N}{8/\Cr{sigma-cusp-3}+1/33}+k',
        \]
        which, upon using the value of $\Cr{sigma-cusp-3}$ from \eqref{eq:minizer-cupsform-2}, implies that $N < 2$. Thus there is at most one real simple zero.
    \end{enumerate}
   
   Thus, taking $\Cr{zfr0} = 33$ ensures a contradiction in all cases, whether $L(s, \pi \times \pi')$ is entire or not.  Theorem~\ref{thm1:ZFR-cuspforms} thus follows.

    To prove the second claim---namely, when $\pi$ is self-dual, the exponent $m[F:\Q]$ of $3+|t|$ in Theorem~\ref{thm1:ZFR-cuspforms} can be reduced to $[F:\Q]$---we consider the alternative inequality
	\begin{equation*}
		\begin{aligned}
			\mathbb{D}_{\sigma}(\N\kn^{-i\gamma}, \N\kn^{i\gamma}) \le 
			\mathbb{D}_{\sigma}( \N\kn^{-i\gamma}, -a_{\pi}(\kn)) + \mathbb{D}_{\sigma}(-a_{\pi}(\kn), \N\kn^{i\gamma})
			= 2\mathbb{D}_{\sigma}( \N\kn^{-i\gamma},-a_{\pi}(\kn)), 
		\end{aligned}
	\end{equation*}
    which holds when $\pi$ is self-dual. 
    Squaring both sides of the above inequality, we have
	\begin{equation*}
		\begin{aligned}
			\mathbb{D}_{\sigma}(\N\kn^{-i\gamma}, \N\kn^{i\gamma})^2 \le 
			4\mathbb{D}_{\sigma}( \N\kn^{-i\gamma},-a_{\pi}(\kn))^2.
		\end{aligned}
	\end{equation*}   
   By the computation in \eqref{eq:gen-distance-2}, the above inequality can be rewritten as
    \begin{equation*}
		0 \le -\frac{\zeta_F'}{\zeta_F}(\sigma) -2\frac{L'}{L}(\sigma, \pi\times \tilde{\pi}) -4\Real\Big(\frac{L'}{L}(\sigma+i\gamma, \pi )\Big)-\Real\Big(\frac{\zeta_F'}{\zeta_F}(\sigma+2i\gamma)\Big). 
    \end{equation*}
     Following a similar analysis as in Section~\ref{subsubsec:Thm1-case2}, we arrive at
	\begin{multline*}
			0 
			\le \frac{3}{\sigma-1}+ \frac{4(\sigma-1)}{(\sigma-1)^2 +\gamma^2}+ \frac{\sigma-1}{(\sigma-1)^2 +4\gamma^2} - \frac{4}{\sigma-\beta}+(8-4\Cr{hadamard})\\
            + (2m+3)\log(\mathfrak{C}(\pi) (3+|\gamma|)^{[F:\Q]}).
		\end{multline*}  
        Since the inequality above differs from \eqref{eq:cusp-case2} only in the exponent of  $3+|\gamma|$ in the last term, our claim follows using the same reasoning as in Section~\ref{subsubsec:Thm1-case2}.

    \subsection{Proof of Theorem~\ref{thm1:ZFR-rankin-selberg}}\label{subsec:thm2}
	
	Let $\pi\in\mathfrak{F}^*_m$ and $\pi' \in\mathfrak{F}^*_{m'}$, and let $\pi \neq\tilde{\pi}$ and $\pi'=\tilde{\pi}'$.  Let $\rho = \beta +i\gamma$ be a nontrivial zero of $L(s, \pi \times \pi')$. 
    Assume that for any constant $\Cl[abcon]{zfr1} > 0$, we have
    \begin{equation}\label{eq:assumption-beta-RS1}
		\begin{aligned}
			\beta
			\ge 1- \frac{1}{\Cr{zfr1}\mathcal{L}_2}, \qquad \mathcal{L}_2 = 2(m+m')\log(\mathfrak{C}(\pi)\mathfrak{C}(\pi') (3+|\gamma|)^{m[F:\Q]}).
		\end{aligned}
	\end{equation}
	We will show that \eqref{eq:assumption-beta-RS1} does not hold for sufficiently large $\Cr{zfr1}$. 
    
    Consider the triangle inequality
	\begin{align*}
		\mathbb{D}_{\sigma}(a_{\pi}(\kn)\N\kn^{-i\gamma}, \overline{a_{\pi}(\kn)}\N\kn^{i\gamma}) &\le \mathbb{D}_{\sigma}(a_{\pi}(\kn)\N\kn^{-i\gamma}, -a_{\pi'}(\kn)) +
		\mathbb{D}_{\sigma}(-a_{\pi'}(\kn), \overline{a_{\pi}(\kn)}\N\kn^{i\gamma})\\
		&= 2\mathbb{D}_{\sigma}(a_{\pi}(\kn)\N\kn^{-i\gamma}, -\overline{a_{\pi'}(\kn)}),
	\end{align*}
	using that $\pi' = \tilde{\pi}'$ in the last step. Squaring both sides of the above inequality, we obtain
	\begin{equation}
		\label{eq:triangle-ineq}
		0\le 4\mathbb{D}_{\sigma}(a_{\pi}(\kn)\N\kn^{-i\gamma}, -\overline{a_{\pi'}(\kn)})^2-\mathbb{D}_{\sigma}(a_{\pi}(\kn)\N\kn^{-2i\gamma}, \overline{a_{\pi}(\kn)})^2.
	\end{equation}
	By \eqref{eq:gen-distance-2}, it follows from \eqref{eq:triangle-ineq} that
    \begin{multline}\label{eq:rankin-4}
			0 \leq -\frac{L'}{L}(\sigma,\pi\times\tilde{\pi})-2\frac{L'}{L}(\sigma,\pi'\times\tilde{\pi}')-4\Real\Big(\frac{L'}{L}(\sigma+i\gamma,\pi\times\pi')\Big)\\
            - \Real\Big(\frac{L'}{L}(\sigma+2i\gamma,\pi\times\pi)\Big).
	\end{multline}

	For the first term of \eqref{eq:rankin-4}, we apply Lemma~\ref{lem:stirling} and use \eqref{eq:positivity} to drop the contribution of all nontrivial zeros. It follows that
	\begin{equation}\label{eq:case1-term1}
		\begin{aligned}
			-\frac{L'}{L}(\sigma, \pi \times \tilde{\pi}) 
            \le \frac{1}{\sigma-1}
			+\Big(1-\frac{\Cr{hadamard}}{2}\Big) + \frac{1}{2}\log \mathfrak{C}(\pi \times \tilde{\pi}).
		\end{aligned}
	\end{equation}
	A similar conclusion can be derived for the second term of \eqref{eq:rankin-4}, namely, 
	\begin{equation}\label{eq:case1-term2}
		\begin{aligned}
			-2\frac{L'}{L}(\sigma, \pi' \times \tilde{\pi}') \le \frac{2}{\sigma-1} + (2-\Cr{hadamard})+ \log \mathfrak{C}(\pi' \times \tilde{\pi}').
		\end{aligned}
	\end{equation}
	For the third term of \eqref{eq:rankin-4}, we apply Lemma~\ref{lem:stirling} and then \eqref{eq:positivity} to drop the contribution of all nontrivial zeros except for $\rho = \beta+i\gamma$. It follows that
	\begin{equation}\label{eq:case1-term4}
		\begin{aligned}
			- 4\Real\Big(\frac{L'}{L}(\sigma + i\gamma, \pi \times \pi') \Big) \le -\frac{4}{\sigma-\beta} -2\Cr{hadamard} +  2\log \mathfrak{C}(i\gamma, \pi \times \pi').
		\end{aligned}
	\end{equation}
	For the last term of \eqref{eq:rankin-4}, using Lemma~\ref{lem:stirling} and \eqref{eq:positivity} and the fact that $\pi \neq \tilde{\pi}$, we obtain
    \begin{equation}\label{eq:case1-term3}
		\begin{aligned}
			-\Real\Big(\frac{L'}{L}(\sigma + 2i\gamma, \pi \times \pi) \Big) \le  -\frac{\Cr{hadamard}}{2} + \frac{1}{2}\log \mathfrak{C}(2i\gamma, \pi \times \pi).
		\end{aligned}
	\end{equation} 
   Inserting \eqref{eq:case1-term1}--\eqref{eq:case1-term3} in \eqref{eq:rankin-4} and applying Lemma~\ref{lem:brumley} and that $3-4\Cr{hadamard} < 0$, we have
	\begin{equation}\label{eq:case1-combined}
		\begin{aligned}
			0 &\le \frac{3}{\sigma-1}-\frac{4}{\sigma-\beta} + \mathcal{L}_2
		\end{aligned}
	\end{equation}
    Let $\Cl[abcon]{sigma1}>0$ denote a sufficiently large constant, and choose
    $\sigma = 1 + {1}/({\Cr{sigma1}\mathcal{L}_2})$.     
    Using \eqref{eq:assumption-beta-RS1} and the choice of $\sigma$, we contradict \eqref{eq:case1-combined} when
    \[
    \Cr{zfr1} >  \min_{\Cr{sigma1} > 1} \frac{3\Cr{sigma1}^2+\Cr{sigma1}}{\Cr{sigma1}-1} =  13.9282\ldots
    \]
    In conclusion, Theorem~\ref{thm1:ZFR-rankin-selberg} follows once we take $\Cr{zfr1} = 14$.
	
	\subsection{Proof of Theorem~\ref{thm2:ZFR-rankin-selberg}}\label{subsec:thm3}

	Let $L(s, \pi \times \pi') = L(s, \tilde{\pi} \times \tilde{\pi}')$ and 
    $\rho = \beta +i\gamma$ be a nontrivial zero of $L(s, \pi \times \pi')$. Assume that for any constant $\Cl[abcon]{zfr2} > 0$, we have
    \begin{equation}\label{eq:assumption-beta-RS2}
		\begin{aligned}
			\beta
			\ge 1- \frac{1}{\Cr{zfr2}\mathcal{L}_3}, \qquad \mathcal{L}_3 = 2(m+m')\log(\mathfrak{C}(\pi) \mathfrak{C}(\pi')(3+|\gamma|)^{m[F:\Q]}).
		\end{aligned}
	\end{equation}
	We will show that \eqref{eq:assumption-beta-RS2} does not hold for sufficiently large $\Cr{zfr2}$.
    
    We begin with the triangle inequality
	\begin{equation}
		\label{eq:triangle-ineq-2.1}
		\mathbb{D}_{\sigma}(a_{\pi}(\kn)\N\kn^{-i\gamma}, a_{\pi}(\kn)\N\kn^{i\gamma}) \le \mathbb{D}_{\sigma}(a_{\pi}(\kn)\N\kn^{-i\gamma}, -\overline{a_{\pi'}(\kn)}) +
		\mathbb{D}_{\sigma}(-\overline{a_{\pi'}(\kn)}, a_{\pi}(\kn)\N\kn^{i\gamma}).
	\end{equation}
    Squaring both sides of \eqref{eq:triangle-ineq-2.1} and applying the inequality of arithmetic and geometric means, we obtain
    \begin{equation*}
        \label{eq:triangle-ineq-2.2}
        \mathbb{D}_{\sigma}(a_{\pi}(\kn)\N\kn^{-i\gamma}, a_{\pi}(\kn)\N\kn^{i\gamma})^2 \le 2\mathbb{D}_{\sigma}(a_{\pi}(\kn)\N\kn^{-i\gamma}, -\overline{a_{\pi'}(\kn)})^2 +
		2\mathbb{D}_{\sigma}(-\overline{a_{\pi'}(\kn)}, a_{\pi}(\kn)\N\kn^{i\gamma})^2.
    \end{equation*}
    Using \eqref{eq:gen-distance-2}  and  that $L(s, \pi \times \pi') = L(s, \tilde{\pi} \times \tilde{\pi}')$, we express the above inequality as
	\begin{multline}
		\label{eq:rankin-6}
			0 \le -\frac{L'}{L}(\sigma, \pi \times \tilde{\pi}) -2\frac{L'}{L}(\sigma, \pi' \times \tilde{\pi}')  - 4\Real\Big(\frac{L'}{L}(\sigma + i\gamma, \pi \times \pi') \Big)\\
            -   \Real\Big(\frac{L'}{L}(\sigma + 2i\gamma, \pi \times \tilde{\pi}) \Big).
	\end{multline}
   We then consider the following two cases.
    \subsubsection{Case 1:  \texorpdfstring{$L(s,\pi\times \pi')$}{en} is entire}\label{subsubsec:entire}
	In this case, the first, second, and third terms of \eqref{eq:rankin-6} can be estimated exactly as in \eqref{eq:case1-term1}, \eqref{eq:case1-term2}, and \eqref{eq:case1-term4}, respectively.  
	For the fourth term of \eqref{eq:rankin-6}, since $L(s, \pi \times \tilde{\pi})$ has a pole at $s=1$, upon using Lemma~\ref{lem:stirling} and \eqref{eq:positivity}, we have
	\begin{equation}\label{eq:case3-term3}
		\begin{aligned}
			-\Real\Big(\frac{L'}{L}(\sigma + 2i\gamma, \pi \times \tilde{\pi}) \Big) 
			&\le \frac{\sigma-1}{(\sigma-1)^2+4\gamma^2} + \Big(1-\frac{\Cr{hadamard}}{2}\Big)+ \frac{1}{2}\log \mathfrak{C}(2i\gamma, \pi \times \tilde{\pi}).
		\end{aligned}
	\end{equation}
    Let $\Cl[abcon]{sigma2}$ denote a sufficiently large positive constant. 
    We consider the following three scenarios. 
    \begin{enumerate}[ leftmargin=1cm]
        \item\label{thm3-case1.1} Suppose that $|\gamma| \ge {2}/({\Cr{sigma2}\mathcal{L}_3})$ and choose $\sigma = 1+ 1/({\Cr{sigma2}\mathcal{L}_3})$. 
        Then $|\gamma| \ge 2(\sigma-1)$, and the first term on the right-hand side of \eqref{eq:case3-term3} is bounded above by $\frac{52}{17(\sigma-1)}$.
	Inserting \eqref{eq:case1-term1}--\eqref{eq:case1-term4} and \eqref{eq:case3-term3} in \eqref{eq:rankin-6} and applying Lemma~\ref{lem:brumley} and the fact that $4-4\Cr{hadamard} <0$, we have
	\begin{equation}\label{eq:4.3.1}
	    0 \le \frac{52}{17(\sigma-1)}-\frac{4}{\sigma-\beta}+\mathcal{L}_3.
	\end{equation}
     Using \eqref{eq:assumption-beta-RS2} and our choice of $\sigma$, we contradict \eqref{eq:4.3.1} once taking 
     \begin{equation}\label{eq:minimizer-rankin-1}
         \Cr{zfr2} > \min_{\Cr{sigma2} > 17/16} \frac{52\Cr{sigma2}^2+17\Cr{sigma2}}{16\Cr{sigma2}-17} = 15.8663\ldots,
     \end{equation}
        where the minimum is attained at $\Cr{sigma2} = 2.2775\ldots$
   

    \item  Suppose that $0 < |\gamma| \le {2}/({\Cr{sigma2}\mathcal{L}_3})$.
	Since $\rho$ is a zero of $L(s, \pi \times \pi')$ and  $L(s, \pi \times \pi') = L(s, \tilde{\pi} \times \tilde{\pi}')$, it follows that $\overline{\rho}$ and $\rho$ are distinct zeros of  $L(s, \pi \times \pi')$. Using Lemma~\ref{lem:stirling} and \eqref{eq:positivity}, we can bound the third term in \eqref{eq:rankin-6} more tightly, namely,
	\begin{multline}\label{eq:case3.3-term5}
			- 4\Real\Big(\frac{L'}{L}(\sigma + i\gamma, \pi \times \pi') \Big) \le -\frac{4}{\sigma-\beta}-\frac{4(\sigma-\beta)}{(\sigma-\beta)^2+4\gamma^2}-2\Cr{hadamard}
            + 2\log \mathfrak{C}(i\gamma, \pi \times \pi').
	\end{multline}
    For the fourth term of \eqref{eq:rankin-6}, it follows from  \eqref{eq:case3-term3} and the fact that $|\gamma| >0$ that
    \begin{equation}\label{eq:case3.2-term4.1}
		-\Real\Big(\frac{L'}{L}(\sigma + 2i\gamma, \pi \times \tilde{\pi}) \Big) \le  \frac{1}{\sigma-1} + \Big(1-\frac{\Cr{hadamard}}{2}\Big) + \frac{1}{2}\log \mathfrak{C}(2i\gamma, \pi \times \pi).
	\end{equation}
    By inserting \eqref{eq:case1-term1}, \eqref{eq:case1-term2}, \eqref{eq:case3.3-term5}, and 
    \eqref{eq:case3.2-term4.1} in \eqref{eq:rankin-6} and applying Lemma~\ref{lem:brumley} and the fact that $4-4\Cr{hadamard} < 0$, it follows that
    \begin{equation*}
        0 \le \frac{4}{\sigma-1}-\frac{4}{\sigma-\beta} - \frac{4(\sigma-\beta)}{(\sigma-\beta)^2+4\gamma^2} + \mathcal{L}_3. 
    \end{equation*}
    Choosing $\sigma = 1 + 4/(\Cr{sigma2}\mathcal{L}_3)$, we have $|\gamma| < \frac{1}{2}(\sigma - \beta)$. It follows that
    \begin{equation}\label{eq:thm3-case1.2}
    0 \le \frac{4}{\sigma-1}-\frac{6}{\sigma-\beta} + \mathcal{L}_3, 
    \end{equation}
   Using \eqref{eq:assumption-beta-RS2}, our choice of $\sigma$, and the value of $\Cr{sigma2}$ from \eqref{eq:minimizer-rankin-1}, we contradict \eqref{eq:thm3-case1.2} when  \[
   \Cr{zfr2} > \frac{\Cr{sigma2}^2+\Cr{sigma2}}{2\Cr{sigma2}-4} = 13.4489\ldots
   \]
	
    \item Suppose that $\rho$ is a multiple zero. Using Lemma~\ref{lem:stirling} and \eqref{eq:positivity}, we then have a tighter bound for the third term of \eqref{eq:rankin-6}, namely,
	\begin{equation}\label{eq:case3.2-term3}
		\begin{aligned}
			- 4\Real\Big(\frac{L'}{L}(\sigma + i\gamma, \pi \times \pi') \Big) \le -\frac{8}{\sigma-\beta} -2\Cr{hadamard}+ 2\log \mathfrak{C}(i\gamma, \pi \times \pi').
		\end{aligned}
	\end{equation}
	Inserting \eqref{eq:case1-term1},\eqref{eq:case1-term2}, \eqref{eq:case3.2-term4.1}, and \eqref{eq:case3.2-term3} in \eqref{eq:rankin-6} and applying Lemma~\ref{lem:brumley} and the fact that $4-4\Cr{hadamard} < 0$, we have
	\begin{equation}\label{eq:thm3-case1.3}
	    0 \le \frac{4}{\sigma-1}-\frac{8}{\sigma-\beta}+ \mathcal{L}_3.  
	\end{equation}
    By choosing $\sigma = 1+1/(\Cr{sigma2}\mathcal{L}_3)$ and using \eqref{eq:assumption-beta-RS2} and the value of $\Cr{sigma2}$ from \eqref{eq:minimizer-rankin-1}, it follows that \eqref{eq:thm3-case1.3} fails when taking \[
    \Cr{zfr2} > \frac{4\Cr{sigma2}^2+\Cr{sigma2}}{4\Cr{sigma2}-1} = 2.8391\ldots
    \]
	\end{enumerate}

    \subsubsection{Case 2:  
    \texorpdfstring{$L(s, \pi \times \pi')$}{L} is not entire}\label{subsubsec:not-entire}
    This means $\pi = \tilde{\pi}'$ and, equivalently, $\tilde{\pi} = \pi'$. Then $L(s, \pi \times \pi')$ has a simple pole at $s =1$. Let $\rho = \beta +i\gamma$ be a nontrivial zero of $L(s, \pi \times \pi')$.
    Using the fact that $\pi = \tilde{\pi}'$ in \eqref{eq:rankin-6}, we obtain
	\begin{equation}
		\label{eq:rankin-7}
		\begin{aligned}
			0 &\le -3\frac{L'}{L}(\sigma, \pi \times \pi') - 4\Real\Big(\frac{L'}{L}(\sigma + i\gamma, \pi \times \pi') \Big) -   \Real\Big(\frac{L'}{L}(\sigma + 2i\gamma, \pi \times \pi') \Big).
		\end{aligned}
	\end{equation}
    For the first term of \eqref{eq:rankin-7}, upon using Lemma~\ref{lem:stirling} and \eqref{eq:positivity}, we have
	\begin{equation}\label{eq:case4-term1}
		\begin{aligned}
			-3\frac{L'}{L}(\sigma, \pi \times \pi') \le \frac{3}{\sigma-1} + \Big(3-\frac{3\Cr{hadamard}}{2}\Big)+ \frac{3}{2}\log \mathfrak{C}(\pi \times \pi').
		\end{aligned}
	\end{equation}
	Similarly, for the second term of \eqref{eq:rankin-7}, we have
	\begin{multline}\label{eq:case4-term2}
			- 4\Real\Big(\frac{L'}{L}(\sigma + i\gamma, \pi \times \pi') \Big) \le \frac{4(\sigma-1)}{(\sigma-1)^2+\gamma^2} -\frac{4}{\sigma-\beta} +(4-2\Cr{hadamard})
            + 2\log \mathfrak{C}(i\gamma, \pi \times \pi').
	\end{multline}
	Lastly, for the third term of \eqref{eq:rankin-7}, we obtain
	\begin{equation}\label{eq:case4-term3}
		\begin{aligned}
			-\Real\Big(\frac{L'}{L}(\sigma + 2i\gamma, \pi \times \pi') \Big) &\le  \frac{\sigma-1}{(\sigma-1)^2+4\gamma^2} +\Big(1-\frac{\Cr{hadamard}}{2}\Big)+ \frac{1}{2}\log \mathfrak{C}(2i\gamma, \pi \times \pi').
		\end{aligned}
	\end{equation}
	Putting together \eqref{eq:case4-term1}--\eqref{eq:case4-term3} in \eqref{eq:rankin-7} and applying Lemma~\ref{lem:brumley} along with the fact that $\pi' = \tilde{\pi}$, we obtain
	\begin{equation}
		\label{eq:rankin-9}
		0 \le \frac{3}{\sigma-1} +\frac{4(\sigma-1)}{(\sigma-1)^2+\gamma^2} + \frac{\sigma-1}{(\sigma-1)^2+4\gamma^2}  -\frac{4}{\sigma-\beta}
        +  (8-4\Cr{hadamard}) + \mathcal{L}_3.  
	\end{equation}
    Let $\Cl[abcon]{sigma3}>0$ denote a sufficiently large constant. 
    We consider the following three scenarios.
    \begin{enumerate}[ leftmargin=1cm]
        \item\label{thm3-case2.1}  Suppose that
        $|\gamma| \ge {4}/({\Cr{sigma3}\mathcal{L}_3})$. 
        Using the definition of $\mathcal{L}_3$ in \eqref{eq:assumption-beta-RS2}, the last two terms of \eqref{eq:rankin-9} can be bounded above by
    \[
    \Big(\frac{8-4\Cr{hadamard}}{4\log 3} + 1 \Big) \mathcal{L}_3 \le 1.2531\mathcal{L}_3 := \kappa\mathcal{L}_3. 
    \]
    Choosing $\sigma = 1 + {1}/({\Cr{sigma3}\mathcal{L}_3})$, we have $\gamma \ge 4(\sigma-1)$. It follows from \eqref{eq:rankin-9} that
    \begin{equation}\label{eq:rankin-9-2}
        0 \le \frac{3592}{1105(\sigma-1)} -\frac{4}{\sigma-\beta}
        +  \kappa\mathcal{L}_3.
    \end{equation}
    Using \eqref{eq:assumption-beta-RS2} and our choice of $\sigma$, we contradict \eqref{eq:rankin-9-2} when 
    \begin{equation}\label{eq:minimizer-rankin-2}
        \Cr{zfr2} >  \min_{\Cr{sigma3} > 1105\kappa/828} \frac{3592\Cr{sigma3}^2+1105\kappa\Cr{sigma3}}{828\Cr{sigma3}-1105\kappa} =  32.2770\ldots,
    \end{equation}
    where the minimum is attained when $\Cr{sigma3} = 3.5273\ldots$
    
    \item\label{thm3-case2.2} Suppose that $0 < |\gamma| < {4}/({\Cr{sigma3}\mathcal{L}_3})$. We consider $-(L'/L)(\sigma, \pi \times \pi')$ and apply Lemma~\ref{lem:stirling} and  \eqref{eq:positivity}. Since $L(s,\pi \times \pi')$ is self-dual and $\gamma \neq 0$, then $\overline{\rho}$ is another zero of $L(s,\pi \times \pi')$. It follows that 
    \begin{equation}\label{eq:thm3-case2.2}
        -\frac{L'}{L}(\sigma, \pi \times \pi') \le \frac{1}{\sigma-1} - \frac{2(\sigma-\beta)}{(\sigma-\beta)^2+\gamma^2} + \Big(1-\frac{\Cr{hadamard}}{2}\Big)+ \frac{1}{2}\log \mathfrak{C}(\pi \times \pi').
    \end{equation}
        By the definition of $\mathcal{L}_3$, the last two terms of \eqref{eq:thm3-case2.2} can be bounded above by
        \[
        \Big(\frac{2-\Cr{hadamard}}{8\log 3} + \frac{1}{8}\Big) \mathcal{L}_3 \le 0.1567\mathcal{L}_3 := \kappa'\mathcal{L}_3. 
        \]
         Since $\pi' = \tilde{\pi}$, then $-(L'/L)(\sigma, \pi \times \pi') \ge 0$ when $\sigma > 1$. Choosing $\sigma = 1 + {8}/({\Cr{sigma3}\mathcal{L}_3})$, it follows that  $|\gamma| < \frac{1}{2}|\sigma-\beta|$. It then follows from \eqref{eq:thm3-case2.2} that
        \begin{equation}\label{eq:rankin-11}
         0 \le \frac{1}{\sigma-1} -\frac{8}{5(\sigma-\beta)}
        +  \kappa'\mathcal{L}_3.
        \end{equation}
        Using \eqref{eq:assumption-beta-RS2}, our choice of $\sigma$, and the value of $\Cr{sigma3}$ from \eqref{eq:minimizer-rankin-2}, we contradict \eqref{eq:rankin-11} when 
        \[
        \Cr{zfr2} > \frac{5\Cr{sigma3}^2+40\kappa'\Cr{sigma3}}{24\Cr{sigma3}-320\kappa'} =  2.4431\ldots
        \]
        Hence, taking $\Cr{zfr2} = 33$, we have proved a zero-free region for Cases \eqref{thm3-case2.1} and \eqref{thm3-case2.2}.
    
        \item 
        
        Lastly, we handle potential real zeros of $L(s, \pi \times \pi')$.
        Let $N$ be the number of all real zeros, counted with multiplicity, in the range $1- 1/(33\mathcal{L}_3) \le \beta \le 1$. Again, we consider $-(L'/L)(\sigma, \pi \times \pi')$ and use Lemma~\ref{lem:stirling}, \eqref{eq:positivity}, and the fact that $-(L'/L)(\sigma, \pi \times \pi') \ge 0$ when $\sigma >1$. It follows that
        \[
        0 \le \frac{1}{\sigma-1}-\frac{N}{\sigma-\beta}+\kappa'\mathcal{L}_3.
        \]
        Using the same choice of $\sigma$ as in Case \eqref{thm3-case2.2}, we obtain
        \[
        0 \le \frac{\Cr{sigma3}}{8}-\frac{N}{8/\Cr{sigma3}+1/33}+\kappa',
        \]
        which, upon using the value of $\Cr{sigma3}$ from \eqref{eq:minimizer-rankin-2}, implies that $N < 2$. Thus there is at most one real simple zero.
    \end{enumerate}
    In conclusion, taking $\Cr{zfr2} = 33$ ensures a contradiction in all cases, whether $L(s, \pi \times \pi')$ is entire or not.  Theorem~\ref{thm2:ZFR-rankin-selberg} thus follows.

    \subsection{Analysis of the metric limitation}\label{sec:limitation-proof}

We now extend our discussion from Section~\ref{subsec:limitations-intro} by demonstrating that when $\pi$, $\pi'$, and $L(s,\pi \times \pi')$ are non-self-dual, this metric framework cannot yield a classical zero-free region. Specifically, we explain why the criterion \eqref{eq:c_P-c_Z-condition} fails in this setting.

For $j \in \{1, 2, 3\}$, we define $x_j(\mathfrak{n}) := \delta_j a_{\pi_j}(\mathfrak{n}) \mathrm{N}\mathfrak{n}^{i\gamma_j}$, where $\gamma_j, \delta_j  \in \mathbb{R}$. Allowing $\delta_j$ to be arbitrary real numbers, rather than restricting them to $\{\pm 1\}$, optimizes flexibility while preserving the metric structure. We begin with an arbitrary triangle inequality of the form \eqref{eq:triangle-D-sigma}, which, via the inequality of arithmetic and geometric means, yields 
\begin{equation}\label{eq:triangle-limit}
    \mathbb{D}_{\sigma}(x_1(\mathfrak{n}), x_2(\mathfrak{n}))^2 \le 2\mathbb{D}_{\sigma}(x_1(\mathfrak{n}), x_3(\mathfrak{n}))^2 + 2\mathbb{D}_{\sigma}(x_3(\mathfrak{n}), x_2(\mathfrak{n}))^2. 
\end{equation}
Assume for simplicity that $\N\mathfrak{q}_{\pi}=\N\mathfrak{q}_{\pi'}=1$. Then $\pi_1,\pi_2,\pi_3$ must also have level 1; otherwise, their corresponding Rankin--Selberg convolutions cannot be isomorphic to $L(s,\pi\times\pi')$, and hence do not contribute to the target zero coefficient $c_Z$ in \eqref{eq:desired form-bino}. We further adopt the standard assumption that $L(s,\pi\times\pi')$ is irreducible, namely that it does not factor into a product of lower-degree $L$-functions unless $\pi' \cong \widetilde{\pi}$.

Applying the identity \eqref{eq:gen-distance-3} to \eqref{eq:triangle-limit}, we obtain
\begin{multline*}
0 \le -\frac{1}{2} \delta_1^2\frac{L'}{L}(\sigma, \pi_1\times\widetilde{\pi}_1) - 2\delta_3^2\frac{L'}{L}(\sigma, \pi_3\times\widetilde{\pi}_3)-\frac{1}{2} \delta_2^2\frac{L'}{L}(\sigma, \pi_2\times\widetilde{\pi}_2)\\
+2\delta_1\delta_3\Real\Big(\frac{L'}{L}(\sigma +(\gamma_3-\gamma_1)i, \pi_1 \times \widetilde{\pi}_3)\Big)
+2\delta_2\delta_3\Real\Big(\frac{L'}{L}(\sigma +(\gamma_2-\gamma_3)i, \pi_3 \times \widetilde{\pi}_2)\Big)\\
-\delta_1\delta_2\Real\Big(\frac{L'}{L}(\sigma +(\gamma_2-\gamma_1)i, \pi_1 \times \widetilde{\pi}_2)\Big).
\end{multline*}
Crucially, if any two of the convolutions $L(s, \pi_1 \times \widetilde{\pi}_3)$, $L(s, \pi_3 \times \widetilde{\pi}_2)$, or $L(s, \pi_1 \times \widetilde{\pi}_2)$ were to coincide with $L(s, \pi \times \pi')$, it would structurally force at least one of $\pi$, $\pi'$, or $L(s, \pi \times \pi')$ to be self-dual. Because we are in a fully non-self-dual setting, at most one of these cross-term $L$-functions can equal $L(s, \pi \times \pi')$. Consequently, by Lemma~\ref{lem:stirling}, the maximum possible value for the zero-contribution coefficient $c_Z$ is restricted to either $2\delta_1\delta_3$, $2\delta_2\delta_3$, or $-\delta_1\delta_2$ while the polar contribution coefficient satisfies $c_P \ge \frac{1}{2}\delta_1^2 + 2\delta_3^2 + \frac{1}{2}\delta_2^2$. It can be verified that none of these potential values for $c_Z$ can strictly exceed $c_P$, completing the proof of the limitation.
    
	

	
    \section{Proofs of Theorems~\ref{thm1:ZFR-cuspforms}--\ref{thm2:ZFR-rankin-selberg} when \texorpdfstring{$\N\kq_{\pi} > 1$}{N2} or \texorpdfstring{$\N\kq_{\pi'} > 1$}{N3}}\label{sec:proof-level>1}

        
    In the previous section, we adopted the family of metrics defined in \eqref{eq:gen-distance} to prove Theorems~\ref{thm1:ZFR-cuspforms}--\ref{thm2:ZFR-rankin-selberg} under the assumption that $\N\kq_\pi  = \N\kq_{\pi'} =1$. If we assume that $\N\kq_{\pi} > 1$ or $\N\kq_{\pi'} > 1$, the square of \eqref{eq:gen-distance} 
    can be expressed as in \eqref{eq:gen-distance-3}, which includes additional contributions from ramified primes of the form  $E(\sigma+i\gamma, \pi_1 \times \pi_2)$. 
    To estimate these extra contributions, we can apply Lemma~\ref{lem:general-error}, which relies on progress towards the generalized Ramanujan conjecture. However, this results in narrower zero-free regions than those in Theorems~\ref{thm1:ZFR-cuspforms}--\ref{thm2:ZFR-rankin-selberg}. For instance, the zero-free region in Theorem~\ref{thm1:ZFR-cuspforms} would take the form
    \begin{equation*}
        \sigma < 1- \frac{1}{\Cr{zfr0}\big[m\log(\mathfrak{C}(\pi)(3+|t|)^{m[F:\Q]})+\frac{m^2}{1-2\theta_{m}}\omega(\kq_{\pi})\big]},
    \end{equation*}
    for some absolute constant $\Cr{zfr0}$.
   The extra term in the denominator, which comes from the application of Lemma~\ref{lem:general-error},  introduces a significant dependence on $m,m'$ and $\kq_{\pi}$.
    
    To eliminate this extra term,  
    we define a new quantity as follows. For integers $m_1, m_2 \ge 1$,  let $\pi_1 \in \mathfrak{F}_{m_1}^*$, $\pi_2 \in \mathfrak{F}_{m_2}^*$, $\gamma_1,\gamma_2 \in \mathbb{R}$, and $\delta_1, \delta_2 \in \{\pm1\}$. We define
    \begin{multline}\label{eq:def-d-leveln}
    \mathbb{D}^*_{\sigma} ((\pi_1, \gamma_1, \delta_1), (\pi_2, \gamma_2, \delta_2))\\
    := \sqrt{\frac{1}{2}\Big[-\frac{L'}{L}(\sigma, \pi_1 \times \tilde{\pi}_1) -\frac{L'}{L}(\sigma, \pi_2 \times \tilde{\pi}_2) +2\delta_1\delta_2 \Real\Big(\frac{L'}{L}(\sigma+(\gamma_2-\gamma_1)i, \pi_1 \times \tilde{\pi}_2) \Big)\Big]},
		\end{multline} 
    where $\sigma$ is a parameter to be optimized later.
    The square of this quantity is essentially identical to \eqref{eq:gen-distance-3}, except for the exclusion of the terms of the form $E(\sigma+i\gamma, \pi_1 \times \pi_2)$. In other words, when $\N\kq_\pi  = \N\kq_{\pi'} =1$ and $\pi_1, \pi_2 \in \{\pi, \pi', \tilde{\pi}, \tilde{\pi}'\}$, we have
    \[
    \mathbb{D}^*_{\sigma} ((\pi_1, \gamma_1, \delta_1), (\pi_2, \gamma_2, \delta_2))^2 = \mathbb{D}_{\sigma}(\delta_1a_{\pi_1}(\kn)\N\kn^{i\gamma_1},\delta_2 a_{\pi_2}(\kn)\N\kn^{i\gamma_2})^2.
    \]

    Our motivation for considering the quantity in \eqref{eq:def-d-leveln} is to eliminate contributions from ramified primes at the outset, ensuring that they do not appear in any triangle inequalities of interest.  For positive integers $m_1, m_2, m_3$, let $\pi_1 \in \mathfrak{F}_{m_1}^*$, $\pi_2 \in \mathfrak{F}_{m_2}^*$, $\pi_3 \in \mathfrak{F}_{m_3}^*$, $\gamma_1, \gamma_{2}, \gamma_{3} \in \mathbb{R}$, and $\delta_1, \delta_{2}, \delta_{3} \in \{\pm1\}$. 
  In what follows, we shall establish that $\mathbb{D}^*_{\sigma}((\pi_1, \gamma_1, \delta_1), (\pi_2, \gamma_2, \delta_2))$ is well-defined and non-negative, and that the following inequality holds:
  \begin{multline}\label{eq:triangle-d*}
   \mathbb{D}^*_{\sigma}((\pi_1, \gamma_1, \delta_1), (\pi_2, \gamma_2, \delta_2))^2\\ \le 2\mathbb{D}^*_{\sigma}((\pi_1, \gamma_1, \delta_1), (\pi_3, \gamma_3, \delta_3))^2 + 2\mathbb{D}^*_{\sigma}((\pi_3, \gamma_3, \delta_3), (\pi_2, \gamma_2,\delta_2))^2.
\end{multline}
\begin{rek}
    In Section~\ref{sec:proof-level1}, we used the triangle inequality \eqref{eq:triangle-D-sigma} to construct a non-negative linear combination of logarithmic derivatives of $L$-functions. In this section, we instead rely on  \eqref{eq:triangle-d*}, which is not a triangle inequality in the usual sense. In fact, when $ \mathrm{N}\mathfrak{q}_\pi = \mathrm{N}\mathfrak{q}_{\pi'} = 1$, \eqref{eq:triangle-d*} is equivalent to squaring \eqref{eq:triangle-D-sigma} and applying the arithmetic--geometric mean inequality. This weaker form, however,  suffices to construct the desired non-negative linear combinations.
\end{rek}
    To this end, we exploit the positive semi-definiteness of Rankin--Selberg $L$-functions. Following the approach in \cite{LP}, we first define positive semi-definite families of $L$-functions.

   \begin{defi}\label{def:nonneg-def}
Let $\mathcal{I}$ be a finite ordered sequence. For $i, j \in \mathcal{I}$, let $L_{i,j}(s) = \sum_{\kn} a_{i,j}(\kn)\, \N\kn^{-s}$ be a formal Dirichlet series with complex coefficients.
We say that the family $(L_{i,j}(s))_{i,j \in \mathcal{I}}$ is \emph{positive semi-definite} if and only if for any $\kn$, the matrix $M \in \mathbb{C}^{\mathcal{I} \times\mathcal{I}}$ with entries
\[
    M_{i, j} := a_{i, j}(\kn)
\]
is Hermitian and positive semi-definite. When applied to complex $L$-functions, this definition refers to their Dirichlet expansions in
$\Re(s) > \sigma$, for large enough $\sigma$.
\end{defi}
    \begin{prop}\label{prop:Lichtman-Pascadi} Recall that $\mathfrak{F}_m$ is the set of all cuspidal automorphic representations $\pi$ of $\GL_m(\mathbb{A}_{F})$ with unitary central character. For any finite subset $\mathfrak{F} \subseteq \cup_{m=1}^{\infty} \mathfrak{F}_m$, the family $( -(L'/L)(s, \pi \times \tilde{\pi}'))_{\pi,\pi' \in \mathfrak{F}}$ is positive semi-definite. In particular, for any $\omega_{\pi}, \omega_{\pi'} \in \mathbb{C}$ and $\sigma > 1$, we have
    \begin{equation*}
        \sum_{\pi, \pi' \in \mathfrak{F}} -\omega_{\pi}\overline{\omega}_{\pi'} \frac{L'}{L}(\sigma,\pi \times\tilde{\pi}') \ge 0.
    \end{equation*}
    \end{prop}
    \begin{proof}
        The proof follows closely from \cite[Proposition 4]{LP}, which establishes that for $F= \mathbb{Q}$, if $m \ge 1$ is an integer and $\mathfrak{F}$ is a finite subset of $\mathfrak{F}_m$, then the family $(\log L(s, \pi \times \tilde{\pi}'))_{\pi, \pi' \in \mathfrak{F}}$ is positive semi-definite. We first observe that the proof of \cite[Proposition 4]{LP} extends naturally to any number field $F$ with minimal modification. Furthermore, since their argument does not depend on $\pi$ and $\pi'$ having the same degree $m$, the same result holds for any finite subset $\mathfrak{F} \subseteq \cup_{m=1}^{\infty} \mathfrak{F}_m$. 

        It remains to show that if the family $(\log L(s, \pi \times \tilde{\pi}'))_{\pi, \pi' \in \mathfrak{F}}$ is positive semi-definite, then so is the family $(-({L'}/{L})(s, \pi \times \tilde{\pi}'))_{\pi, \pi' \in \mathfrak{F}}$. Given $\sigma > 1$, observe that for each ideal $\mathfrak{n}$, the Dirichlet coefficients at $\mathfrak{n}$ of $\log L(\sigma, \pi \times \tilde{\pi}')$ and of $-({L'}/{L})(\sigma, \pi \times \tilde{\pi}')$ differ only by a positive scalar that depends only on $\mathfrak{n}$, and is independent of $\pi$ and $\pi'$. Therefore, by Definition~\ref{def:nonneg-def}, the positive semi-definiteness of $(-(L'/L)(s, \pi \times \tilde{\pi}'))_{\pi, \pi' \in \mathfrak{F}} $ follows directly from that of $( \log L(s, \pi \times \tilde{\pi}'))_{\pi, \pi' \in \mathfrak{F}} $.
         \end{proof}
        
   
    Recall $\mathbb{D}^*_{\sigma} ((\pi_1, \gamma_1, \delta_1), (\pi_2, \gamma_2, \delta_2))$ in \eqref{eq:def-d-leveln}. To show that this quantity is always a non-negative real number, it suffices to verify the inequality
   \begin{equation*}
        \frac{1}{2}\Big[-\frac{L'}{L}(\sigma, \pi_1 \times \tilde{\pi}_1) -\frac{L'}{L}(\sigma, \pi_2 \times \tilde{\pi}_2) +2\delta_1\delta_2 \Real\Big(\frac{L'}{L}(\sigma+(\gamma_2-\gamma_1)i, \pi_1 \times \tilde{\pi}_2) \Big)\Big] \ge 0.
    \end{equation*}
    Recall that $\delta_1, \delta_2 \in \{\pm1\}$. The above inequality holds by setting $\mathfrak{F}=\{\pi_1 \otimes |\cdot|^{-i\gamma_1}, \pi_2 \otimes |\cdot|^{-i\gamma_2}\}$, with the corresponding constants $\omega_{\pi_1} = \delta_{1}$ and $\omega_{\pi_2} = -\delta_{2}$, respectively, and then applying Proposition~\ref{prop:Lichtman-Pascadi}. 
    
    Next, we verify \eqref{eq:triangle-d*}. 
    By Definition \eqref{eq:def-d-leveln}, this is equivalent to verifying that
\begin{multline*}
        0 \le \frac{1}{2}\Big[-\frac{L'}{L}(\sigma, \pi_1 \times\tilde{\pi}_1)  -\frac{L'}{L}(\sigma, \pi_2 \times \tilde{\pi}_2)-4\frac{L'}{L}(\sigma, \pi_3 \times \tilde{\pi}_3)\\
        +4\delta_1{\delta}_3\Real\Big(\frac{L'}{L}(\sigma+(\gamma_3-\gamma_1)i, \pi_1\times \tilde{\pi}_3)\Big)+4\delta_{3}{\delta}_{2}\Real\Big(\frac{L'}{L}(\sigma+(\gamma_{2}-\gamma_3)i, \pi_3 \times \tilde{\pi}_2)\Big)\\
        - 2\delta_1{\delta}_2 \Real\Big(\frac{L'}{L}(\sigma+(\gamma_2-\gamma_1)i, \pi_1 \times \tilde{\pi}_2)\Big)\Big].
\end{multline*}
Recall that $\delta_1, \delta_2, \delta_3 \in \{\pm1\}$. Setting $\mathfrak{F}=\{\pi_1 \otimes |\cdot|^{-i\gamma_1}, \pi_2 \otimes |\cdot|^{-i\gamma_2}, \pi_3 \otimes |\cdot|^{-i\gamma_3}\}$ with the corresponding constants $\omega_{\pi_1} = \delta_{1}$, $\omega_{\pi_2} = \delta_{2}$, and $\omega_{\pi_3} = -2\delta_{3}$, respectively, we conclude that the above inequality holds upon applying Proposition~\ref{prop:Lichtman-Pascadi}.

\begin{proof}[Proofs of Theorems~\ref{thm1:ZFR-cuspforms}--\ref{thm2:ZFR-rankin-selberg}.]
We begin by proving Theorem~\ref{thm1:ZFR-cuspforms}. Recalling the hypothesis of Theorem~\ref{thm1:ZFR-cuspforms}, we consider the inequality
\begin{equation}\label{eq:thm1-N>1}
\mathbb{D}^*_{\sigma}((\pi,-\gamma,1), (\tilde{\pi},\gamma,1))^2 \le 2\mathbb{D}^*_{\sigma}((\pi,-\gamma,1), (\mathbbm{1},0,-1))^2 + 2\mathbb{D}^*_{\sigma}((\mathbbm{1},0,-1), (\tilde{\pi},\gamma,1))^2,
\end{equation}
which follows from \eqref{eq:triangle-d*}.
By the definition of $\mathbb{D}^*_{\sigma} ((\pi_1,\gamma_1,\delta_1), (\pi_2,\gamma_2,\delta_2))^2$ in \eqref{eq:def-d-leveln}, the inequality \eqref{eq:thm1-N>1} simplifies to
\begin{equation*}
		0 \le -\frac{L'}{L}(\sigma, \pi\times \tilde{\pi})-2\frac{\zeta_F'}{\zeta_F}(\sigma) -\Real\Big(\frac{L'}{L}(\sigma+2i\gamma, \pi \times \pi)\Big)-4\Real\Big(\frac{L'}{L}(\sigma+i\gamma, \pi)\Big).
\end{equation*}
This recovers \eqref{eq:thm1-eq1}, which serves as the starting point for the proof of Theorem~\ref{thm1:ZFR-cuspforms} under the assumption $\N\kq_{\pi}=\N\kq_{\pi'}=1$. 
From this point, the proof proceeds identically to Section~\ref{subsec:proof-level1-thm1}.

Similarly, for Theorem~\ref{thm1:ZFR-rankin-selberg}, we consider the inequality
\begin{equation}\label{eq:leveln-thm2}
\begin{aligned}
\mathbb{D}^*_{\sigma}((\pi,-\gamma,1), (\tilde{\pi},\gamma,1))^2 &\le 2\mathbb{D}^*_{\sigma}((\pi,-\gamma,1), (\pi',0,-1))^2 + 2\mathbb{D}^*_{\sigma}((\pi',0,-1), (\tilde{\pi},\gamma,1))^2\\
&\le 4\mathbb{D}^*_{\sigma}((\pi,-\gamma, 1), (\pi',0,-1))^2,
\end{aligned}
\end{equation}
using that $\pi'$ is self-dual the last step.
By Definition \eqref{eq:def-d-leveln}, the inequality \eqref{eq:leveln-thm2} recovers \eqref{eq:rankin-4}, which serves as the starting point for the proof of Theorem~\ref{thm1:ZFR-rankin-selberg} under the assumption 
$\N\kq_{\pi}=\N\kq_{\pi'}=1$.  The rest of the proof then follows the same steps as in Section~\ref{subsec:thm2}.

Lastly, to prove Theorem~\ref{thm2:ZFR-rankin-selberg}, we consider the inequality 
\begin{equation}\label{eq:leveln-thm3}
\mathbb{D}^*_{\sigma}((\pi,-\gamma,1), (\pi,\gamma,1))^2 \le 2\mathbb{D}^*_{\sigma}((\pi,-\gamma,1), (\tilde{\pi}',0,-1))^2 + 2\mathbb{D}^*_{\sigma}((\tilde{\pi}',0,-1), (\pi,\gamma,1))^2.    
\end{equation}
Again, by \eqref{eq:def-d-leveln}, the inequality \eqref{eq:leveln-thm3} recovers \eqref{eq:rankin-6}, which serves as the starting point for the proof of Theorem~\ref{thm2:ZFR-rankin-selberg} under the assumption
$\N\kq_{\pi}=\N\kq_{\pi'}=1$. The rest of the proof follows the same steps as in Section~\ref{subsec:thm3}.
\end{proof}


\begin{rek}
    The method in this section actually supersedes the work in Section~\ref{sec:proof-level1}, which is restricted to the case $\N\kq_{\pi}=\N\kq_{\pi'}=1$. We included Section~\ref{sec:proof-level1}, which uses the family of metrics  $\mathbb{D}_{\sigma}(\delta_1a_{\pi_1}(\kn)\N\kn^{i\gamma_1},\delta_2 a_{\pi_2}(\kn)\N\kn^{i\gamma_2})$, to illustrate how Koukoulopoulos’s work can be extended in a similar manner.
\end{rek}




\begin{appendices}
\section{An inequality on the Rankin--Selberg conductor}
    We now restate \eqref{eq:conductor-ineq} in the following lemma and provide a detailed proof, which closely follows the argument in \cite[Theorem A.2]{Humphries-Brumley} and incorporates the correction noted on \cite[P. 9]{Ramakrishnan-Yang}.
    \begin{lem}\label{lem:brumley}
        If $(\pi,\pi')\in\mathfrak{F}_{m}^*\times\mathfrak{F}_{m'}^*$, then
        \[
        \mathfrak{C}(it, \pi\times \pi') \le  \mathfrak{C}(\pi)^{m'}\mathfrak{C}(\pi')^m(|t|+3)^{mm'[F:\Q]}.
        \]
    Furthermore, we have $\mathfrak{C}(\pi \times \pi') \le \mathfrak{C}(\pi)^{m'}\mathfrak{C}(\pi')^{m}.$  
    \end{lem}
    \begin{proof}
	For the arithmetic conductor, we use the bound $\N\kq_{\pi\times{\pi'}}\mid \N\kq_{\pi}^{m'}\N\kq_{{\pi'}}^{m}$ as given in \cite[Theorem 2]{BushHen}. For the archimedean conductor,  it remains to show that
    \[
    \mathfrak{C}_{v}(it, \pi \times \pi') \le \mathfrak{C}_{v}( \pi )^{m'}\mathfrak{C}_{v}( \pi')^{m}(|t|+3)^{mm'[F_{v}:\mathbb{R}]}.
     \]
    We note that the proof of this lemma closely follows \cite[Theorem A.2]{Humphries-Brumley}. Utilizing Langlands' classification of the admissible dual of $\GL_{n}(F_v)$,  the representations $\pi_v$ and $\pi'_v$ correspond to the direct sums $\oplus \varphi_i$ and $\oplus\varphi_j'$, for irreducible representations $\varphi_i$ and $\varphi'_j$ of the Weil group $W_{F_v}$ of $F_v$. By definition, we have 
     \[
     L(s,\pi_v) = \prod_i L(s, \varphi_i), \quad L(s,\pi_v') = \prod_j L(s, \varphi_j'), \quad L(s,\pi_v \times \pi_v') = \prod_{i,j} L(s, \varphi_i \times \varphi'_j).
     \]
     This factorization also applies to the associated conductors.   Let $d_i, d'_j$ denote the dimensions of $\varphi_i$ and $\varphi'_j$, respectively, so that $n = \sum d_i$ and $n' = \sum d_j'$. Omitting the indices $i$ and $j$, it remains to prove that for irreducible representations $\varphi$ and $\varphi'$ of $W_{F_v}$, of respective dimensions $d$ and $d'$, we have 
   \begin{equation}\label{eq:goal-archimedean}
       \mathfrak{C}_{v}(it, \varphi \otimes \varphi') \le \mathfrak{C}_{v}(\varphi )^{d'}\mathfrak{C}_{v}(  \varphi')^{d}(|t|+3)^{dd'[F_{v}:\mathbb{R}]}.
   \end{equation}

   When $F_v = \mathbb{C}$, the Weil group is given by $W_{\mathbb{C}} = \mathbb{C}^{\times}$, implying that all irreducible representations are one-dimensional. Any such character can be expressed as $\chi_{k,v}(z) = (z/|z|)^k|z|^{2\nu}$, for $k \in \mathbb{Z}$ and $\nu \in \mathbb{C}$. Defining $\mu = \nu + |k|/2$, the corresponding $L$-factor is $\Gamma_{\mathbb{C}}(s+\mu)$ (see \cite[(4.6)]{Knapp}). According to the recipe in \cite[(21) and (31)]{Iwaniec-Sarnak}, we obtain
   \[
        \mathfrak{C}_v(it, \varphi) = (|\mu+it|+3)^2.
   \]
   Now if $\varphi = \chi_{k,\nu}$ and $\varphi = \chi_{k',\nu'}$, then $\varphi \otimes \varphi' = \chi_{k+k',\nu+\nu'}$. This leads to the bound
   \begin{equation}\label{eq:local-cond}
        \mathfrak{C}_v(it, \varphi \otimes \varphi') = \Big( \Big|\frac{|k+k'|}{2} +\nu+\nu'+it\Big|+3\Big)^2 \le \Big(\Big(\frac{|k|}{2}+|\nu|\Big) + \Big(\frac{|k'|}{2}+|\nu'|\Big)+|t|+3 \Big)^2.
    \end{equation}
   
    Next, we establish that $\frac{|k|}{2} + |\nu| \le 3|\mu|+1$. This part of the proof in \cite[Appendix A.2]{Humphries-Brumley} is slightly incomplete, as the inequality preceding \cite[(A.13)]{Humphries-Brumley} fails when $k=1$. This issue was identified and corrected in \cite[P. 9]{Ramakrishnan-Yang}. We present a complete proof of this part as follows. 
    \begin{enumerate}
        \item If $k=0$, then $\frac{|k|}{2} + |\nu| \le  3|\mu|+1$ holds trivially. 
        \item If $|k|=1$, we have $\frac{|k|}{2} + |\nu| = |\nu|+\frac{1}{2} \le |\nu +\frac{1}{2}|+1 \le 3|\nu +\frac{1}{2}|+1 = 3|\mu|+1$.
        \item If $|k| \ge 2$, we recall \cite[Corollary 2.5]{Jacquet-Shalika}, namely, $|\Real(\nu)| \le \frac{1}{2}$. Since $\frac{|k|}{2} > 1$, we have $|\Real(\nu)| \le |\Real(\nu)+\frac{|k|}{2}|$, and also $2|\Real(\nu)+\frac{|k|}{2}| \ge \frac{|k|}{2}$.  It follows that
        \[
        |\nu| +\frac{|k|}{2} \le \Big|\nu + \frac{|k|}{2}\Big| +\frac{|k|}{2} \le \Big|\nu + \frac{|k|}{2}\Big|+ 2\Big|\Real(\nu)+\frac{|k|}{2}\Big| \le 3\Big|\nu+\frac{|k|}{2}\Big| \le 3|\mu|+1.\]
    \end{enumerate}
    Thus we have shown that $\frac{|k|}{2} + |\nu| \le  3|\mu|+1$. Substituting this into \eqref{eq:local-cond}, we obtain
    \[
        \mathfrak{C}_v(it, \varphi \otimes \varphi') \le (  3|\mu| + 3|\mu'|+|t|+5)^2 \le (|\mu|+3)^2(|\mu'|+3)^2( |t|+3)^2,
   \]
   which implies \eqref{eq:goal-archimedean} in the case where  $F_v = \mathbb{C}$.
   
    When $F_v = \mathbb{R}$, each irreducible representation $\varphi$ of $W_{\R} = \mathbb{C}^{\times} \cup j\mathbb{C}^{\times}$ has dimension 1 or 2. 
    \begin{itemize}
        \item If $\varphi$ is one-dimensional, its restriction to $\mathbb{C}^{\times}$  takes the form $\chi_{0,\nu}$ for $\nu \in \mathbb{C}$ (see \cite{Knapp}, (3.2)). Defining $k =1 -\varphi(j) \in \{0,2\}$, we set $\mu = \nu + k/2$, yielding 
        \[
        L(s, \varphi) = \Gamma_{\mathbb{R}}(s+\mu), \qquad \mathfrak{C}_{v}(it, \varphi) = |\mu+it|+3.
        \]
        \item  If $\varphi$ is two-dimensional, then it is the induction of $\chi_{k,\nu}$ from  $\mathbb{C}^{\times}$ to $\GL_2(\mathbb{R})$, where $k \ge 1$ and $\nu \in \mathbb{C}$. Setting $\mu =\nu  + k/2$, we obtain 
        \[
        L(s, \varphi) = \Gamma_{\mathbb{C}}(s+\mu), \qquad \mathfrak{C}_v(it, \varphi) = (|\mu+it|+3)^2.
        \]
    \end{itemize}
      In both cases, let $(k,\nu)$ and $(k',\nu')$ be the parameters associated with $\varphi$ and $\varphi'$, respectively. We now analyze the  tensor product 
      parameters.
    \begin{enumerate}
        \item If both  $\varphi$ and $\varphi'$ are one-dimensional, then $\varphi \otimes \varphi'$ remains one-dimensional, with parameter $(1-\varphi(j)\varphi'(j), \nu+\nu')$. In this case, \eqref{eq:goal-archimedean} is equivalent to \begin{equation}\label{eq:complex-1}
        \Big|\frac{1-\varphi(j)\varphi'(j)}{2}+\nu+\nu'+it\Big|+3 \le  \Big( \frac{k}{2} +|\nu|\Big) + \Big( \frac{k'}{2} +|\nu'|\Big)+|t| +3.
        \end{equation}
        Since $1-\varphi(j)\varphi'(j) \le (1-\varphi(j))+(1-\varphi'(j)) = k+k'$, we 
        can apply the same analysis as in the complex setting to obtain \eqref{eq:complex-1}.
        \item If $\varphi$ is one-dimensional, and $\varphi'$ is irreducible and two-dimensional, then $\varphi \otimes \varphi'$ is irreducible and two-dimensional, induced from $\mathbb{C}^{\times}$ by the character $\chi_{0,\nu}\chi_{k',\nu'}=\chi_{k',\nu+\nu'}$. Thus $\varphi \otimes \varphi'$ has parameters $(k', \nu +\nu')$. Then \eqref{eq:goal-archimedean} is equivalent to
        \[
        (|\mu+\mu'+it|+3)^2 \le (|\mu|+3)^2(|\mu'|+3)^2(|t|+3)^2,
        \]
        which holds by the triangle inequality.
        \item Let $\varphi$ and $\varphi'$ be both irreducible and two-dimensional, and let $k \ge k'$. Then $\varphi \otimes \varphi'$ is the direct sum of two two-dimensional representations, induced from $\mathbb{C}^{\times}$ from the characters $\chi_{k,\nu}\chi_{k',\nu'} = \chi_{k+k',\nu+\nu'}$ and $\chi_{-k,-\nu}\chi_{k',\nu'} = \chi_{k'-k,\nu'-\nu}$. (The latter representation is reducible when $k=k'$.) This shows that 
            \begin{align*}
                L(s, \varphi \otimes \varphi') &= \Gamma_{\mathbb{C}}(s+\mu+\mu')\Gamma_{\mathbb{C}}(s+\mu-\mu'),\\ 
                \intertext{and}
                \mathfrak{C}_v(it, \varphi \otimes \varphi') &= (|\mu+\mu'|+3)^2(|\mu-\mu'|+3)^2.
            \end{align*}
        Thus \eqref{eq:goal-archimedean} is equivalent to 
        \[
        (|\mu+\mu'+it|+3)^2(|\mu-\mu'+it|+3)^2 \le (|\mu|+3)^4(|\mu'|+3)^4(|t|+3)^4,
        \]
        which holds by the triangle inequality.
    \end{enumerate}
    This completes the proof of Lemma~\ref{lem:brumley}.
    \end{proof}
\end{appendices}

\bibliographystyle{abbrv}
\bibliography{ref-positive}

@article {ST,
    AUTHOR = {Soundararajan, Kannan and Thorner, Jesse},
     TITLE = {Weak subconvexity without a {R}amanujan hypothesis},
      NOTE = {With an appendix by Farrell Brumley},
   JOURNAL = {Duke Math. J.},
  FJOURNAL = {Duke Mathematical Journal},
    VOLUME = {168},
      YEAR = {2019},
    NUMBER = {7},
     PAGES = {1231--1268},
      ISSN = {0012-7094,1547-7398},
   MRCLASS = {11F67 (11F66 11M41)},
  MRNUMBER = {3953433},
MRREVIEWER = {Henry H. Kim},
       DOI = {10.1215/00127094-2018-0065},
       URL = {https://doi.org/10.1215/00127094-2018-0065},
}

@book {D,
    AUTHOR = {Koukoulopoulos, Dimitris},
     TITLE = {The distribution of prime numbers},
    SERIES = {Graduate Studies in Mathematics},
    VOLUME = {203},
 PUBLISHER = {American Mathematical Society},
    ADDRESS = {Providence, RI},
      YEAR = {[2019] \copyright 2019},
     PAGES = {xii + 356},
      ISBN = {978-1-4704-4754-0; 978-1-4704-6285-7},
   MRCLASS = {11N05 (11-01 11M06 11N35 11N60)},
  MRNUMBER = {3971232},
MRREVIEWER = {Y.-F. S. P\'etermann},
       DOI = {10.1090/gsm/203},
}

@book {IK,
    AUTHOR = {Iwaniec, Henryk and Kowalski, Emmanuel},
     TITLE = {Analytic number theory},
    SERIES = {American Mathematical Society Colloquium Publications},
    VOLUME = {53},
 PUBLISHER = {American Mathematical Society},
    ADDRESS = {Providence, RI},
      YEAR = {2004},
     PAGES = {xii+615},
      ISBN = {0-8218-3633-1},
   MRCLASS = {11-02 (11Fxx 11Lxx 11Mxx 11Nxx)},
  MRNUMBER = {2061214},
MRREVIEWER = {K.\ Soundararajan},
       DOI = {10.1090/coll/053},
}

@article {MS,
    AUTHOR = {M\"{u}ller, W. and Speh, B.},
     TITLE = {Absolute convergence of the spectral side of the {A}rthur
              trace formula for {${\rm GL}_n$}},
      NOTE = {With an appendix by E. M. Lapid},
   JOURNAL = {Geom. Funct. Anal.},
  FJOURNAL = {Geometric and Functional Analysis},
    VOLUME = {14},
      YEAR = {2004},
    NUMBER = {1},
     PAGES = {58--93},
      ISSN = {1016-443X},
   MRCLASS = {22E55 (11F70 11F72)},
  MRNUMBER = {2053600},
MRREVIEWER = {Werner Hoffmann},
       DOI = {10.1007/s00039-004-0452-0},
}

@article {JPS,
    AUTHOR = {Jacquet, H. and Piatetski-Shapiro, I. I. and Shalika, J. A.},
     TITLE = {Rankin--{S}elberg convolutions},
   JOURNAL = {Amer. J. Math.},
  FJOURNAL = {American Journal of Mathematics},
    VOLUME = {105},
      YEAR = {1983},
    NUMBER = {2},
     PAGES = {367--464},
      ISSN = {0002-9327,1080-6377},
   MRCLASS = {11F67 (11F70 11R39 22E55)},
  MRNUMBER = {701565},
MRREVIEWER = {Freydoon\ Shahidi},
       DOI = {10.2307/2374264},
}

@Inbook{LRS,
    AUTHOR = {Luo, Wenzhi and Rudnick, Ze\'ev and Sarnak, Peter},
     TITLE = {On the generalized {R}amanujan conjecture for {${\rm GL}(n)$}},
 BOOKTITLE = {Automorphic forms, automorphic representations, and arithmetic
              ({F}ort {W}orth, {TX}, 1996)},
    SERIES = {Proc. Sympos. Pure Math.},
    VOLUME = {66, Part 2},
     PAGES = {301--310},
 PUBLISHER = {Amer. Math. Soc.},
    ADDRESS = {Providence, RI},
      YEAR = {1999},
      ISBN = {0-8218-0659-9},
   MRCLASS = {11F70 (11F72)},
  MRNUMBER = {1703764},
MRREVIEWER = {Stefan\ K\"uhnlein},
       DOI = {10.1090/pspum/066.2/1703764},
}

@article {BushHen,
    AUTHOR = {Bushnell, C. J. and Henniart, G.},
     TITLE = {An upper bound on conductors for pairs},
   JOURNAL = {J. Number Theory},
  FJOURNAL = {Journal of Number Theory},
    VOLUME = {65},
      YEAR = {1997},
    NUMBER = {2},
     PAGES = {183--196},
      ISSN = {0022-314X,1096-1658},
   MRCLASS = {11S37 (11F70 22E50)},
  MRNUMBER = {1462836},
MRREVIEWER = {Ernst-Wilhelm Zink},
       DOI = {10.1006/jnth.1997.2142},
}

@unpublished{GS,
  author = {Granville, Andrew and Soundararajan, Kannan},
  title = {Multiplicative number theory},
  note = {Snowbird MRC notes (unpublished)},
  year = {2011},
}

@ARTICLE{LP,
       author = {{Duker Lichtman}, Jared and {Pascadi}, Alexandru},
        title = "{Density theorems for $\text{GL}_n$ via Rankin--Selberg $L$-functions}",
      journal = {arXiv e-prints},
    year = 2024,
    month = aug,
    pages = {arXiv:2408.13682},
    archivePrefix = {arXiv},
    eprint = {2408.13682},
    primaryClass = {math.NT},
    adsurl = {https://ui.adsabs.harvard.edu/abs/2024arXiv240813682D}
}

@book {GJ,
    AUTHOR = {Godement, Roger and Jacquet, Herv\'e},
     TITLE = {Zeta functions of simple algebras},
    SERIES = {Lecture Notes in Mathematics},
    VOLUME = {Vol. 260},
 PUBLISHER = {Springer-Verlag},
ADDRESS = {Berlin-New York},
      YEAR = {1972},
     PAGES = {ix+188},
   MRCLASS = {12A80 (12A65 12A70 12B35 22E55)},
  MRNUMBER = {342495},
MRREVIEWER = {L.\ Corwin},
}

@article {MW,
    AUTHOR = {M\oe glin, C. and Waldspurger, J.-L.},
     TITLE = {Le spectre r\'esiduel de {${\rm GL}(n)$}},
   JOURNAL = {Ann. Sci. \'Ecole Norm. Sup. (4)},
  FJOURNAL = {Annales Scientifiques de l'\'Ecole Normale Sup\'erieure.
              Quatri\`eme S\'erie},
    VOLUME = {22},
      YEAR = {1989},
    NUMBER = {4},
     PAGES = {605--674},
      ISSN = {0012-9593},
   MRCLASS = {22E55 (11F70 22E50)},
  MRNUMBER = {1026752},
MRREVIEWER = {Stephen\ Gelbart},
       URL = {http://www.numdam.org/item?id=ASENS_1989_4_22_4_605_0},
}

@article {moreno,
    AUTHOR = {Moreno, Carlos J.},
     TITLE = {Analytic proof of the strong multiplicity one theorem},
   JOURNAL = {Amer. J. Math.},
  FJOURNAL = {American Journal of Mathematics},
    VOLUME = {107},
      YEAR = {1985},
    NUMBER = {1},
     PAGES = {163--206},
      ISSN = {0002-9327,1080-6377},
   MRCLASS = {22E55 (11F67 11F70 11S40)},
  MRNUMBER = {778093},
MRREVIEWER = {David\ Soudry},
       DOI = {10.2307/2374461},
}

@article {Humphries-Brumley,
    AUTHOR = {Humphries, Peter and Brumley, Farrell},
     TITLE = {Standard zero-free regions for {R}ankin--{S}elberg
              {$L$}-functions via sieve theory},
   JOURNAL = {Math. Z.},
  FJOURNAL = {Mathematische Zeitschrift},
    VOLUME = {292},
      YEAR = {2019},
    NUMBER = {3-4},
     PAGES = {1105--1122},
      ISSN = {0025-5874,1432-1823},
   MRCLASS = {11M26 (11F66 11N36)},
  MRNUMBER = {3980284},
MRREVIEWER = {A.\ Perelli},
       DOI = {10.1007/s00209-018-2136-8},
}

@article {Humphries-Thorner,
    AUTHOR = {Humphries, Peter and Thorner, Jesse},
     TITLE = {Towards a {${\rm GL}_n$} variant of the {H}oheisel phenomenon},
   JOURNAL = {Trans. Amer. Math. Soc.},
  FJOURNAL = {Transactions of the American Mathematical Society},
    VOLUME = {375},
      YEAR = {2022},
    NUMBER = {3},
     PAGES = {1801--1824},
      ISSN = {0002-9947,1088-6850},
   MRCLASS = {11F66},
  MRNUMBER = {4378080},
MRREVIEWER = {Liyang\ Yang},
       DOI = {10.1090/tran/8544},
}

@article {HIJT,
    AUTHOR = {Hoey, Alexandra and Iskander, Jonas and Jin, Steven and Trejos
              Su\'arez, Fernando},
     TITLE = {An unconditional explicit bound on the error term in the
              {S}ato--{T}ate conjecture},
   JOURNAL = {Q. J. Math.},
  FJOURNAL = {The Quarterly Journal of Mathematics},
    VOLUME = {73},
      YEAR = {2022},
    NUMBER = {4},
     PAGES = {1189--1225},
      ISSN = {0033-5606,1464-3847},
   MRCLASS = {11F30 (11F80 11G05)},
  MRNUMBER = {4520218},
MRREVIEWER = {Fei\ Hou},
       DOI = {10.1093/qmath/haac004},
}

@incollection{Knapp,
    AUTHOR = {Knapp, A. W.},
     TITLE = {Local {L}anglands correspondence: the {A}rchimedean case},
 BOOKTITLE = {Motives ({S}eattle, {WA}, 1991)},
    SERIES = {Proc. Sympos. Pure Math.},
    VOLUME = {55, Part 2},
     PAGES = {393--410},
 PUBLISHER = {Amer. Math. Soc.},
    ADDRESS = {Providence, RI},
      YEAR = {1994},
      ISBN = {0-8218-1637-3},
   MRCLASS = {11F70 (11R39 11S37 22E50)},
  MRNUMBER = {1265560},
MRREVIEWER = {Stephen\ Gelbart},
       DOI = {10.1090/pspum/055.2/1265560},
}

@incollection {Iwaniec-Sarnak,
    AUTHOR = {Iwaniec, H. and Sarnak, P.},
     TITLE = {Perspectives on the analytic theory of {$L$}-functions},
      NOTE = {GAFA 2000 (Tel Aviv, 1999)},
   JOURNAL = {Geom. Funct. Anal.},
  FJOURNAL = {Geometric and Functional Analysis},
      YEAR = {2000},
     PAGES = {705--741},
      ISSN = {1016-443X,1420-8970},
   MRCLASS = {11M26 (11F67 11F70 11F72 11G40)},
  MRNUMBER = {1826269},
MRREVIEWER = {Henry\ H.\ Kim},
       DOI = {10.1007/978-3-0346-0425-3\_6},
}

@article {Jacquet-Shalika,
    AUTHOR = {Jacquet, Herv\'e{} and Shalika, Joseph A.},
     TITLE = {A non-vanishing theorem for zeta functions of {${\rm
              GL}\sb{n}$}},
   JOURNAL = {Invent. Math.},
  FJOURNAL = {Inventiones Mathematicae},
    VOLUME = {38},
      YEAR = {1976/77},
    NUMBER = {1},
     PAGES = {1--16},
      ISSN = {0020-9910,1432-1297},
   MRCLASS = {12A70 (10D20 10H10)},
  MRNUMBER = {432596},
MRREVIEWER = {Mark\ E.\ Novodvorsky},
       DOI = {10.1007/BF01390166},
}

@article {Ramakrishnan-Yang,
    AUTHOR = {Ramakrishnan, Dinakar and Yang, Liyang},
     TITLE = {A constraint for twist equivalence of cusp forms on {${\rm
              GL}(n)$}},
   JOURNAL = {Funct. Approx. Comment. Math.},
  FJOURNAL = {Uniwersytet im. Adama Mickiewicza w Poznaniu. Wydzia\l\
              Matematyki i Informatyki. Functiones et Approximatio
              Commentarii Mathematici},
    VOLUME = {65},
      YEAR = {2021},
    NUMBER = {1},
     PAGES = {105--117},
      ISSN = {0208-6573,2080-9433},
   MRCLASS = {11F11 (11F55 11F70)},
  MRNUMBER = {4311499},
       DOI = {10.7169/facm/1913},
}

@article{Sato-Tate-2,
       author = {{Thorner}, Jesse},
        title = "{Exceptional zeros of Rankin--Selberg $L$-functions and joint Sato--Tate distributions}",
      journal = {arXiv e-prints},
    year = 2024,
    month = apr,
    pages = {arXiv:2404.06482},
    archivePrefix = {arXiv},
    eprint = {2404.06482},
    primaryClass = {math.NT},
       adsurl = {https://ui.adsabs.harvard.edu/abs/2024arXiv240406482T}
}

@article {GL,
    AUTHOR = {Goldfeld, Dorian and Li, Xiaoqing},
     TITLE = {A standard zero free region for {R}ankin--{S}elberg
              {$L$}-functions},
   JOURNAL = {Int. Math. Res. Not. IMRN},
  FJOURNAL = {International Mathematics Research Notices. IMRN},
      YEAR = {2018},
    NUMBER = {22},
     PAGES = {7067--7136},
      ISSN = {1073-7928,1687-0247},
   MRCLASS = {11M41 (11F03 11F30 11F66)},
  MRNUMBER = {3878594},
MRREVIEWER = {A.\ Perelli},
       DOI = {10.1093/imrn/rnx087},
}

@article {Leung,
    AUTHOR = {Leung, Sun-Kai},
     TITLE = {A note on the standard zero-free region for {$L$}-functions},
   JOURNAL = {Expo. Math.},
  FJOURNAL = {Expositiones Mathematicae},
    VOLUME = {42},
      YEAR = {2024},
    NUMBER = {6},
     PAGES = {Paper No. 125624, 7},
      ISSN = {0723-0869,1878-0792},
   MRCLASS = {11M41 (11F66)},
  MRNUMBER = {4813549},
MRREVIEWER = {A.\ Perelli},
       DOI = {10.1016/j.exmath.2024.125624},
       URL = {https://doi-org.proxy2.library.illinois.edu/10.1016/j.exmath.2024.125624},
}

@ARTICLE{Tatuzawa,
       author = {{Harcos}, Gergely and {Thorner}, Jesse},
        title = "{Tatuzawa's theorem for Rankin-Selberg $L$-functions}",
      journal = {arXiv e-prints},
         year = 2025,
        month = aug,
          eid = {arXiv:2508.10844},
        pages = {arXiv:2508.10844},
          doi = {10.48550/arXiv.2508.10844},
archivePrefix = {arXiv},
       eprint = {2508.10844},
 primaryClass = {math.NT},
       adsurl = {https://ui.adsabs.harvard.edu/abs/2025arXiv250810844H}
}

@article {Harcos-Thorner,
    AUTHOR = {Harcos, Gergely and Thorner, Jesse},
     TITLE = {A new zero-free region for {R}ankin-{S}elberg {$L$}-functions},
   JOURNAL = {J. Reine Angew. Math.},
  FJOURNAL = {Journal f\"ur die Reine und Angewandte Mathematik. [Crelle's
              Journal]},
    VOLUME = {822},
      YEAR = {2025},
     PAGES = {179--201},
      ISSN = {0075-4102,1435-5345},
   MRCLASS = {11F66 (11F70)},
  MRNUMBER = {4899236},
MRREVIEWER = {Guodong\ Hua},
       DOI = {10.1515/crelle-2025-0009},
       URL = {https://doi-org.proxy2.library.illinois.edu/10.1515/crelle-2025-0009},
}

\end{document}